\documentclass[a4paper,13pt,abstracton]{article}
 \usepackage[utf8]{inputenc}
   
   \usepackage{amsmath,amsthm,verbatim,amssymb,amsfonts,amscd, graphicx,mathrsfs}
 \usepackage{geometry}
 \usepackage{tabularx}
 \usepackage{diagbox}
 \usepackage{pifont}
 \numberwithin{equation}{section}
 \newtheorem{thm}{Theorem}[section]
 \newtheorem{lem}[thm]{Lemma}
 \newtheorem{defi}[thm]{Definition}
 \newtheorem{cor}[thm]{Corollary}
 \newtheorem{prop}[thm]{Proposition}
 \newtheorem{rem}[thm]{Remark}

 \usepackage{enumerate}
 \usepackage[square,comma,super,numbers,sort&compress]{natbib}
 \usepackage{bm}
  \usepackage{hyperref}
 \hypersetup{colorlinks=true, urlcolor=Cerulean, citecolor=red}
 \usepackage[T1]{fontenc}
 \usepackage{nameref,cleveref}
 \usepackage{nicefrac,xfrac}
 \usepackage[all]{xy}
 \usepackage{color}
\usepackage[section]{placeins}
\usepackage{indentfirst}

\begin{document}
\title{Homogeneous Ulrich bundles on isotropic flag varieties} 
\author{Xinyi Fang
	\thanks{Department of Mathematics, Shanghai Normal University, Shanghai, 200234, PR China, xinyif@shnu.edu.cn. The research is supported by National Natural Science Foundation of China (Grant No. 12471040).}
	\ and Yusuke Nakayama
	\thanks{School of Fundamental Science and Engineering, Waseda University, 3-4-1, Okubo, Shinjuku, Tokyo, 169-8555, yusuke216144@akane.waseda.jp.}}
\date{}

	\maketitle
	\begin{abstract}
In this paper, we consider the existence problem of Ulrich bundles on a rational homogeneous space $G/P$ of type $B$, $C$ or $D$. We show that if the Picard number of $G/P$ is greater than or equal to $2$, then 
there are no irreducible homogeneous Ulrich bundles on $G/P$ with respect to the minimal ample class.
	\end{abstract}
	\textbf{MSC}: {Primary 14F05; Secondary 14M17}
	{\flushleft{{\bf Keywords:} irreducible homogeneous bundle, Ulrich bundle, isotropic flag variety}}
	\section{Introduction}
	Ulrich bundles have been studied extensively in algebraic geometry in recent years.
	A vector bundle $\mathcal{E}$ on a complex smooth polarized projective variety $(X,\mathcal{O}_X(1))$ is said to be {\it Ulrich} if $H^{i}(X,\mathcal{E}(-t))$ vanish for all $0\leq i\leq \dim X$ and $1\leq t\leq \dim X$.
	Such bundles were originally introduced by Ulrich \cite{Ulrich} in the context of commutative algebra.
	Indeed, Ulrich bundles on a smooth projective variety $X$ correspond to maximally generated maximal Cohen--Macaulay $R_X$-modules, where $R_X$ is the graded homogeneous coordinate ring of $X$.
	
There is an intriguing problem to investigate whether every smooth projective variety carries an Ulrich bundle with respect to some polarization, as posed by Eisenbud, Schreyer and Weyman \cite{ESW}.
	The problem is still open except for a few cases.
	Eisenbud, Schreyer and Weyman proved that every projective curve and Veronese variety admits an Ulrich sheaf.
	Beauville \cite{Bea2} showed that there are Ulrich line bundles and rank $2$ Ulrich bundles on some surfaces and the del Pezzo threefold.
	Various other studies have also been conducted on some varieties (see \cite{ACCMT, ACMR, AFO, Bea1, BN, Cas, CKM, Fae, Lop1, Lop2, MRPM}).
	
	Let us consider the case that a smooth projective variety is a rational homogeneous space $G/P$, where $G$ is a complex semi-simple linear algebraic group and $P$ is a parabolic subgroup of $G$.
	We first review the case that $G/P$ has Picard number one, i.e., $P$ is a maximal parabolic subgroup.
	Then simple algebraic groups of all types have been studied.
	When $G$ is of type $A$, $G/P$ is the Grassmann variety.
	Costa and Mir\'{o}-Roig \cite{CM1} investigated all irreducible homogeneous Ulrich bundles on Grassmann varieties.
	In the case of type $B,C$ or $D$, Fonarev \cite{Fon} classified irreducible homogeneous Ulrich bundles on isotropic Grassmann varieties.
	In order to do this, a criterion for an irreducible homogeneous vector bundle to be Ulrich applicable to all $G/P$ with Picard number one was given in the paper. Later, Lee and Park \cite{LP} examined the cases of exceptional types.
	In particular, they proved that the only homogeneous varieties $G/P$ with Picard number one admitting an irreducible homogeneous Ulrich bundle are the Cayley plane $E_{6}/P_{1}$ and the $E_{7}$-adjoint variety $E_{7}/P_{1}$.
	
	We consider the case that $G/P$ has Picard number at least $2$ and its polarization is the minimal ample class (see subsection $2.1$ for the definition of minimal ample class.).
	In type $A$, such variety is a $r$-step partial flag variety $F(k_1,\ldots,k_r;n)$, where $k_1,\ldots,k_r,n$ are positive integers such that $k_1<\cdots<k_r<n$.
	Coskun et al. \cite{CCHMW}
	proved that $F(k_1,\ldots,k_r;n)$ does not carry irreducible homogeneous Ulrich bundles with respect to the minimal ample class if $r\geq3$.
	Furthermore, they classified irreducible homogeneous Ulrich bundles with respect to the minimal ample class on certain two-step flag varieties and conjectured that the two-step flag variety $F(k_{1},k_{2};n)$ does not admit irreducible homogeneous Ulrich bundles with respect to the minimal ample class if $k_{1}\geq3$ and $k_{2}-k_{1}\geq3$.
This conjecture has not been completely solved, although Coskun and Jaskowiak \cite{CJ} gave a partially affirmative answer to this conjecture.
Being motivated by the above works, the first author \cite{Fang} and the second author \cite{Nakayama} independently gave a criterion for an irreducible homogeneous vector bundle on $G/P$ of all types with any Picard number to be Ulrich with respect to any polarization.
	These criteria extend Fonarev's result for rational homogeneous spaces with Picard number one.
	Moreover, the second author showed that when $G$ is of type $E_6$, $F_4$ or $G_2$, $G/P$ does not admit irreducible homogeneous Ulrich bundles with respect to the minimal ample class if the Picard number of $G/P$ is greater than or equal to $2$.
	
	In this paper, we study the existence problem of Ulrich bundles with respect to the minimal ample class on isotropic flag varieties.
	Here an isotropic flag variety means a rational homogeneous space $G/P$ of type $B,\ C$ or $D$.
	The main theorem in this paper is the following.
	
	
\begin{thm}[Theorem \ref{mainthm}]\label{thm}Let $G/P$ be an 
isotropic flag variety. If the Picard number of $G/P$ is greater than or equal to $2$, then there are no irreducible homogeneous Ulrich bundles on	
 $G/P$ with respect to the minimal ample class. 
\end{thm}	
	
Combining our result with Fonarev's \cite{Fon},	one obtains a complete classification of irreducible homogeneous bundles which are Ulrich with respect to the minimal ample class on isotropic flag varieties. In order to prove the main theorem, we use the aforementioned criterion.
	In particular, we define the associated datum of an irreducible homogeneous vector bundle and use it to analyze the existence of irreducible homogeneous Ulrich bundles. Here, we mention ACM bundles, an extended class of Ulrich bundles.
	In \cite{CM2},\cite{DFR} and \cite{FNR}, irreducible homogeneous ACM bundles on $G/P$ with Picard number one have been investigated.
	Moreover, the first author \cite{Fang} characterized such bundles on $G/P$ with any Picard number.

\paragraph{Plan of the paper} The paper is organized as follows.
In section $2$, we collect some definitions, notations and restate Theorem \ref{thm} in a precise form.
In particular, we will explain the aforementioned criterion and the associated datum of an irreducible homogeneous vector bundle. In Section $3$, we first give a concrete description of the associated datum on $G/P$, where $G$ is of type $B_n$, $C_n$ or $D_n$. Then we prove the main theorem on a case by case basis.

\paragraph{Notations and conventions}

\begin{itemize}
	\item $X=G/P$: the rational homogeneous space with simple algebraic group $G$ and parabolic subgroup $P$;	
	\item  $\Phi^+$: the set of positive roots;
	\item  $\lambda_i$: the $i$-th fundamental weight;
	\item $\rho$: $\lambda_1+\cdots+\lambda_n$;
	\item $|M|$: the number of entries in
	a set $M$ or a matrix $M$;
	\item $M_{\min}$~(resp. $M_{max}$): the smallest~(resp. largest) entry in an integer matrix $M$;
	\item $M_{uv}$: the element at the $u$-th row and $v$-th column of a matrix $M$;
	\item $lcm(x_1,x_2,\ldots,x_n)$: the least common multiple of integers $x_1,x_2,\ldots,x_n$;
	\item $E_\lambda$: the irreducible homogeneous vector bundle with highest weight $\lambda$;
	\item $T^{\lambda}_{X}$: the associated datum of $E_\lambda$ on $X$.
\end{itemize}
	
\section{Preliminaries}
Throughout this paper, all algebraic varieties and morphisms will be defined over the complex field $\mathbb{C}$.
\subsection{Rational homogeneous spaces}
We will prepare some notations and briefly describe the basic facts about rational homogeneous spaces and the polarizations of these varieties (more details see \cite{Ott} and \cite{Snow}.).
Let $G$ be a semi-simple linear algebraic group over the complex field $\mathbb{C}$ and $T\subset G$ a maximal torus.
We set $\mathfrak{g}:=$Lie$(G)$ and $\mathfrak{h}:=$Lie$(T)$. Let $\Phi$ be the set of roots associated with the pair $(\mathfrak{g},\mathfrak{h})$ and $\Delta:=\{\alpha_{1},\ldots,\alpha_{n}\}\subset \Phi$ a set of fixed simple roots. Let $\Phi^{+}$ be the set of positive roots.
The {\it weight lattice} $\Lambda$ is the set of all linear functions $\lambda:\mathfrak{h}\to\mathbb{C}$ for which $\frac{2(\lambda,\alpha)}{(\alpha,\alpha)}\in \mathbb{Z}$ for any $\alpha \in \Phi$, where $(,)$ denotes the Killing form. 
Denote by $\lambda_{1},\ldots,\lambda_{n}\in \Lambda$ the {\it fundamental weights\/}, i.e., $\frac{2(\lambda_{i},\alpha_{j})}{(\alpha_{j},\alpha_{j})}=\delta_{ij}$.

Let $\{1,2,\ldots,n\}$ be the index set of Dynkin nodes of $G$ and $J$ a subset of $\{1,2,\ldots,n\}$.
Let $\Phi^{-}_J$ be a set consisting of negative roots $\alpha$ with $\alpha=\sum_{j\notin J}a_{j}\alpha_{j}$. Let
$$\mathfrak{p}_{J}:=\mathfrak{h}{\huge \oplus}(\oplus_{\alpha\in\Phi^{+}}\mathfrak{g}_{\alpha})\oplus(\oplus_{\alpha\in\Phi^{-}_{J}}\mathfrak{g}_{\alpha})$$
and $P_{J}$ be the parabolic subgroup of $G$ whose Lie algebra is $\mathfrak{p}_{J}$, where each $\mathfrak{g}_{\alpha}$ is a one-dimensional eigenspace with respect to the adjoint action of $\mathfrak{h}$.
In our conventions, $P_J$ is a maximal parabolic subgroup of $G$ when $J$ is a subset of a single element.
The quotient $G/P_J$ is called a {\it rational homogeneous space}.
For example, $A_n/P_{\{k\}}$ is the usual Grassmann variety $Gr(k,n+1)$.
We call $G/P_J$ an {\it isotropic flag variety} if $G$ is a classical algebraic group of type $B_n$, $C_n$ or $D_n$.

Since every semi-simple linear algebraic group can be decomposed into a direct product of simple linear algebraic groups, every rational homogeneous space $G/P_J$ can be decomposed into a product
$$G/P_J\cong G_1/P_{J_1}\times G_2/P_{J_2}\times \cdots \times G_m/P_{J_m}$$
of rational homogeneous spaces with simple linear algebraic groups $G_i$ and parabolic subgroups $P_{J_i}$.
Therefore, from now on, we only consider the case where $G$ is a simple linear algebraic group.

A rational homogeneous space $G/P_{J}$ has $|J|$ projections $\pi_{j}$ from $G/P_{J}$ to $G/P_{\{ j\}}$, where $|J|$ is the number of elements in $J$.
The Picard  group of $G/P_{J}$ is generated by $L_{j}:=\pi^{*}_{j}\mathcal{O}_{G/P_{j}}(1)$.
Hence the Picard number of $G/P_{J}$ is equal to $|J|$. Let $\otimes_{j\in J}L_{j}^{\otimes b_{j}}~(b_j>0)$ be a very ample line bundle on $G/P_{J}$. Then there is a natural embedding $G/P_{J}\subset \mathbb{P}(V^{*})$ with $V=H^0(G/P_{J},\otimes_{j\in J}L_{j}^{\otimes b_{j}})$. If $b_{j}=1$ for every $j\in J$, we say $G/P_{J}$  in its {\it minimal\ homogeneous\ embedding} and call
$\otimes_{j\in J}L_{j}$ the {\it minimal\ ample\ class}. In this paper, we focus on the isotropic flag varieties in their minimal homogeneous embeddings.

\subsection{Homogeneous vector bundles}
Next, we consider vector bundles on $G/P_{J}$.
In particular, we introduce an important class of vector bundles on this variety.
\begin{defi} Let $E$ be a vector bundle on $G/P_{J}$. $E$ is called {\it homogeneous} if there exists an action of $G$ over $E$ such that the following diagram commutes
	\[
	\begin{CD}
		G\times E @>>> E \\
		@VVV    @VVV \\
		G\times G/P_{J}   @>>>  G/P_{J}.
	\end{CD}
	\]
\end{defi}

It is evident from this definition that the tangent bundle $T(G/P_J)$ is homogeneous.
It is well known that homogeneous vector bundles correspond to representations of parabolic subgroup.

\begin{lem}[\cite{Ott} Theorem $9.7$] A vector bundle $E$ of rank $r$ on $G/P_J$ is homogeneous if and only if there exists a representation $\rho:P_{J}\to GL(r)$ such that $E\cong E_{\rho}$, where $E_{\rho}$ is the associated vector bundle.
\end{lem}
Let $E$ be a homogeneous vector bundle on $G/P_{J}$.
By the above lemma, there is a representation of $P_J$ corresponding to $E$.
If this representation is irreducible, we call $E$ an {\it irreducible\ homogeneous\ vector\ bundle}.
We now describe all irreducible representations of $P_{J}$.
Let $\lambda_{j}~(j\in J)$ be the corresponding fundamental weights and $S_{P_{J}}$ the semi-simple part of $P_{J}$.
Then all irreducible representations of $P_{J}$ are
$$V\otimes \left(\otimes_{j\in J} L_{\lambda_{j}}^{t_{j}}\right),~t_{j}\in \mathbb{Z}$$
where $V$ is an irreducible representation of $S_{P_{J}}$ and $L_{\lambda_{j}}$ is a one-dimensional representation with weight $\lambda_{j}$.
If $\lambda$ is the highest weight of an irreducible representation $V$ of $S_{P_{J}}$, we say that $\lambda+\sum_{j\in J}t_{j}\lambda_{j}$ is the highest weight of an irreducible representation $V\otimes \left(\otimes_{j\in J} L_{\lambda_{j}}^{t_{j}}\right)$ of $P_{J}$.

In this paper, we denote $E_{\lambda}$ by the irreducible homogeneous vector bundle arising from the irreducible representation of $P_{J}$ with highest weight $\lambda$.
Unfortunately, many interesting bundles are not irreducible.
Nevertheless, it is often possible to draw conclusions about an arbitrary homogeneous vector bundle from the irreducible case by considering a filtration (c.f. \cite{Snow} Section 5).
Hence we only consider the irreducible homogeneous vector bundles.

Now that the notations are ready, the main theorem is stated again.
\begin{thm}\label{mainthm}Let $G$ be a linear algebraic group of type $B_n$, $C_n$ or $D_n$ and $P_J$ a parabolic subgroup of $G$ that is not maximal (i.e., $|J|\geq2$).
	Then the isotropic flag variety $G/P_J$ does not admit irreducible homogeneous Ulrich bundles with respect to the minimal ample class. 
\end{thm}
\subsection{Criterion for irreducible homogeneous bundles to be Ulrich}
In this subsection, we review a criterion for an irreducible homogeneous vector bundle on $G/P_{J}$ to be Ulrich with respect to the minimal ample class obtained by the first author \cite{Fang} and the second author \cite{Nakayama}. Let us first recall the definition of an Ulrich bundle.
\begin{defi}
	Let $X\subset \mathbb{P}^N$ be a projective variety and $\mathcal{O}_X(1)$ the corresponding very ample line bundle. We say $E$ is an Ulrich bundle on $X$ if
	\begin{align*}
		\left\{
		\begin{array}{lll}
			h^i(X,E\otimes\mathcal{O}_X(-t))=0, & \text{for}~0< i <\dim X~\text{and all}~t\in \mathbb{Z};\\
			h^0(X,E\otimes\mathcal{O}_X(-t))=0, & \text{for}~t\ge 1;\\
			h^{\dim X}(X,E\otimes\mathcal{O}_X(-t))=0, & \text{for}~t\le \dim X.\\
		\end{array}
		\right.   
	\end{align*}
\end{defi}

Let $X=G/P_{J}\subset \mathbb{P}^N$ be a rational homogeneous space in its minimal homogeneous embedding. Let $\Phi_{J}^{+}$ be a set defined by
$$\Phi_{J}^{+}:=\{\alpha\in\Phi^{+}\ |\ (\lambda_{j},\alpha)\neq 0\ {\rm for\ some}\ j\in J\}.$$ 
Note that the number of elements in $\Phi_{J}^{+}$ is equal to the dimension of $X$.

Let $E_{\lambda}$ be an irreducible homogeneous vector bundle on $X$ with highest weight $\lambda=\sum_{i=1}^{n}a_{i}\lambda_{i}$ and $a_j$ non-negative integer for every $j\notin J$.
We define a map $\varphi^{\lambda}$ 
and its associated datum $T_X^\lambda$ 
as follows:
$$\varphi^{\lambda}:\Phi_{J}^{+} \to \mathbb{Q},\ \  \alpha\mapsto\frac{(\lambda+\rho,\alpha)}{(\sum_{j\in J}\lambda_{j},\alpha)},\ {\rm and}\ T_X^\lambda:= Image(\varphi^{\lambda}),$$
where $\rho$ is the half sum of all positive roots.
We now state the criterion for an irreducible homogeneous vector bundle to be Ulrich with respect to the minimal ample class. 

\begin{thm}[\cite{Fang} Theorem $1.3$,\cite{Nakayama}]\label{FN}With the notations as above, the following are equivalent.\\
	(1) $E_{\lambda}$ is an Ulrich vector bundle with respect to the minimal ample class.\\
	(2) $T_X^\lambda=\{1,2,\ldots,\dim X\}$.
\end{thm}

In this paper, we consider the existence problem of Ulrich bundles with respect to the minimal ample class. Theorem \ref{FN} and \cite[Remark 3.4]{Fang} give us a necessary condition for $E_{\lambda}$ to be Ulrich, that is $a_i\ge 0$ for any $1\le i\le n$.
 We use Theorem \ref{FN} to prove the main theorem is tightly centered around the following criteria. Suppose $E_{\lambda}$ is an irreducible homogeneous Ulrich bundle on $X$ and $T_X^\lambda$ is its associated datum. Then 
\begin{itemize}
	\item All entries of $T_X^\lambda$ must be integers;
	\item All integers between $1$ and $\dim X$ must appear in $T_X^\lambda$;
	\item All entries of $T_X^\lambda$ are distinct.
\end{itemize}

\section{Homogeneous Ulrich bundles on isotropic flag varieties}
Throughout this section, we let $X=G/P_J\subset\mathbb{P}^N$ be an isotropic flag variety in its minimal homogeneous embedding, where $G$ is a simple algebraic group of type $B_n$, $C_n$ or $D_n$ and $J=\{d_1,\ldots,d_s\}\subset \{1,2,\ldots,n\}$. Let $E_{\lambda}$ be an irreducible homogeneous vector bundle on $X$ with highest weight $\lambda=a_1\lambda_1+\cdots+a_n\lambda_n$ and $T_X^{\lambda}$ its associated datum.
In \cite{Fang} Section $3.2$, the author explicitly describes the form of $T_X^{\lambda}$ in terms of $a_i$ for different types of $G$ and different values of $d_s$. Here, we recall them for the sake of self-consistency.

 For convenience, we always denote $\overline{a_k}:=a_k+1~(1\le k\le n)$ and denote $M_{min}$~(resp. $M_{max}$) by the smallest~(resp. largest) entry in a matrix $M$. We will prove that if the Picard number $\rho(X)=|J|=s$ is greater than or equal to $2$, there are no irreducible homogeneous Ulrich bundles on $X$ with respect to the minimal ample class.

\subsection{Type $B_n$ or $C_n$}
When $G$ is of type $B_n$ or $C_n$, we prove the main theorem separately depending on whether $d_s$ is $n$.
\subsubsection{$J=\{d_1,\ldots,d_s\}$ with $d_s\ne n$}\label{section1}

According to \cite[Section 3.2.1 II]{Fang}, when $J$ is a subset of $  \{1,2,\ldots,n\}$ consisting of  $d_1,\ldots,d_s$ with $d_s\ne n$, the associated datum $T_X^{\lambda}$ of $E_{\lambda}$ is of the following form. (By convention we set $d_0=0$ and $d_{s+1}=n$.)\\

 $T_X^{\lambda}=\{P^{ij}, Q^{ij}~(1\le i\le j\le s),R^{i}~(1\le i\le s)\}$, where
\begin{align*}
	&P^{ij}_{uv}=\frac{\sum\limits_{k=d_i-u+1}^{d_j+v-1}\overline{a_k}}{j-i+1}~(1\le u\le d_i-d_{i-1},1\le v\le d_{j+1}-d_j);\\
	&Q^{ij}_{uv}=\frac{\sum\limits_{k=d_{i-1}+u}^{n-1}\overline{a_k}+\sum\limits_{k=d_j+v}^{n-1}\overline{a_k}+2\mathfrak{e}\overline{a_n}}
	{2s+1-(i+j)}
		~(1\le u\le d_i-d_{i-1},1\le v\le d_{j+1}-d_j);\\
	&R^{i}_{uv}=\frac{\sum\limits_{k=d_{i-1}+u}^{n-1}\overline{a_k}+\sum\limits_{k=d_{i-1}+v}^{n-1}\overline{a_k}+2\mathfrak{e}\overline{a_n}}
	{2(s+1-i)}~(1\le u\le v\le d_i-d_{i-1}),  
\end{align*}
and
\[\mathfrak{e}=\left\{\begin{matrix}\frac{1}{2}& \text{if}~G ~\text{is of type} ~B_n,\\
	1& \text{if}~G ~\text{is of type} ~C_n.\\
\end{matrix}\right.\]

 In order to prove the nonexistence of irreducible homogeneous Ulrich bundles on $X$, we first prove the following two lemmas.
 \begin{lem}\label{lemma1} 
 Let $E_{\lambda}$ be an irreducible homogeneous bundle on $X$ with highest weight $\lambda=a_1\lambda_1+\cdots+a_n\lambda_n$. If $E_{\lambda}$ is an Ulrich bundle, then for any integer $1\le i\le s+1$ and $d_{i-1}+1\le k\le d_{i}-1$, we have
 $$lcm(1,2,\ldots, 2s-i+1)\mid\overline{a_k}$$
 and $lcm(1,2,\ldots,s)\mid\mathfrak{e}\overline{a_n}$, where $lcm(x_1,x_2,\ldots,x_n)$ is the least common multiple of integers $x_1,x_2,\ldots,x_n$.
 \end{lem}
\begin{proof}
If $E_{\lambda}$ is an Ulrich bundle, then due to Theorem \ref{FN}, all entries $P^{ij}_{uv}$, $Q^{ij}_{uv}$ and $R^{i}_{uv}$ should be integers and so their differences are also integers. Fix an integer $i~(1\le i\le s)$ and let $j$ be an integer between $i$ and $s$, then we have
$$P^{ij}_{u+1,1}-P^{ij}_{u,1}=\frac{\overline{a_{d_i-u}}}{j-i+1}\in\mathbb{Z}~\text{and}~Q^{ij}_{u,1}-Q^{ij}_{u+1,1}=\frac{\overline{a_{d_{i-1}+u}}}{2s+1-(i+j)}\in\mathbb{Z}$$
for any $1\le u\le d_i-d_{i-1}-1$. Thus we obtain that 
$$j-i+1\mid\overline{a_{d_{i-1}+u}}~\text{and}~2s+1-(i+j)\mid\overline{a_{d_{i-1}+u}}.$$
Since $j$ runs through all integers between $i$ and $s$, for any $1\le u\le d_i-d_{i-1}-1$, we have
\begin{align}\label{lcm1}
lcm(1,2,\ldots,s-i)\mid\overline{a_{d_{i-1}+u}}~\text{and}~lcm(s-i+1,s-i+2,\ldots,2s-2i+1)\mid\overline{a_{d_{i-1}+u}}.
\end{align}
Moreover, we note that 
\[
R^{i}_{u,d_i-d_{i-1}}-R^{i}_{u+1,d_i-d_{i-1}}=\frac{\overline{a_{d_{i-1}+u}}}{2(s-i+1)}\in\mathbb{Z},
\]
hence we have 
\begin{align}\label{lcm2}
	2(s-i+1)\mid\overline{a_{d_{i-1}+u}}.
\end{align}	
Furthermore, if $i\ge 2$, let $i'$ be an integer between $1$ and $i-1$, then we have
\[
Q^{i',i-1}_{1,u}-Q^{i',i-1}_{1,u+1}=\frac{\overline{a_{d_{i-1}+u}}}{2s+1-(i'+i-1)}\in\mathbb{Z}
\]
for any $1\le u\le d_i-d_{i-1}-1$. It follows that $2s+1-(i'+i-1)\mid\overline{a_{d_{i-1}+u}}$. Since $i'$ runs through all integer between $1$ and $i-1$, we get
$lcm(2s-2i+3,2s-2i+4,\ldots,2s-i+1)\mid\overline{a_{d_{i-1}+u}}$. This together with (\ref{lcm1}) and (\ref{lcm2}) gives us 
\begin{align}\label{lcm3}
	lcm(1,2,\ldots, 2s-i+1)\mid\overline{a_k}
\end{align}
for any integer $1\le i\le s$ and $d_{i-1}+1\le k\le d_{i}-1$. Similarly, since 
$$Q^{is}_{1,v}-Q^{is}_{1,v+1}=\frac{\overline{a_{d_s+v}}}{2s+1-(i+s)}\in\mathbb{Z},$$
we obtain that 
$s-i+1\mid\overline{a_{d_s+v}}$ for any integer $1\le v\le n-d_s-1$. Since $i$ can run through all integers between $1$ and $s$, we have
\begin{align}\label{lcm4}
	lcm(1,2,\ldots,s)\mid\overline{a_k}
\end{align}
for any integer $d_s+1\le k\le n-1$. By combining sequences (\ref{lcm3}) and (\ref{lcm4}), we can get the first statement of the lemma. 

Next, we analyze $a_n$. Notice that for any integer $1\le i\le s$, 
\begin{align*}
	Q^{is}_{1,n-d_s}-R^{i}_{11}&=\frac{\sum\limits_{k=d_{i-1}+1}^{n-1}\overline{a_k}+2\mathfrak{e}\overline{a_n}}{2s+1-(i+s)}-\frac{2\sum\limits_{k=d_{i-1}+1}^{n-1}\overline{a_k}+2\mathfrak{e}\overline{a_n}}{2(s+1-i)}\\
	&=\frac{\mathfrak{e}\overline{a_n}}{s+1-i}.
\end{align*} 
Since $Q^{is}_{1,n-d_s}$ and $R^{i}_{11}$ are all integers, their difference is also an integer and hence $(s+1-i)\mid\mathfrak{e}\overline{a_n}$. Since $i$ can run through all integers between $1$ and $s$, we have
\[
lcm(1,2,\ldots,s)\mid\mathfrak{e}\overline{a_n}.
\] 
\end{proof}

\begin{lem}\label{lemma2}
With the notations as above, if $E_{\lambda}$ is an Ulrich bundle, then $\overline{a_{d_1}},\overline{a_{d_2}},\ldots,\overline{a_{d_s}}$ are $s$ different odd numbers.
\end{lem}
\begin{proof}
		According to \cite[Remark 3.4]{Fang}, $\min_{1\le i\le s}\{\overline{a_{d_i}}\}$ is the smallest entry in the associated datum $T_X^{\lambda}$ of $E_{\lambda}$. Since $E_{\lambda}$ is an Ulrich bundle, $T_X^{\lambda}=\{1,2,\ldots, \dim X\}$  and hence $\min_{1\le i\le s}\{\overline{a_{d_i}}\}=1$. If $s=1$, then $\overline{a_{d_1}}=1$ is odd. If $s\ge 2$, then for any integer $1\le i\le s-1$,
	\[
	P^{i,i+1}_{min}=\frac{\sum\limits_{k=d_i}^{d_{i+1}}\overline{a_k}}{2}=\frac{\overline{a_{d_i}}+\overline{a_{d_{i+1}}}+\sum\limits_{k=d_i+1}^{d_{i+1}-1}\overline{a_k}}{2}\in\mathbb{Z},	
	\]
	because $P^{i,i+1}_{min}\in T_X^{\lambda}$. By Lemma \ref{lemma1}, we know that $\overline{a_k}$ is even for any $k\in [d_i+1,d_{i+1}-1]$. Hence $P^{i,i+1}_{min}$ is an integer is equivalent to say that $\overline{a_{d_i}}+\overline{a_{d_{i+1}}}$ is even. Therefore, all integers $\overline{a_{d_i}}~(1\le i\le s)$ have the same parity. Note that $\min_{1\le i\le s}\{\overline{a_{d_i}}\}=1$, which means that there is an integer $t$ such that $\overline{a_{d_t}}$ is equal to $1$. Since $\overline{a_{d_i}}$ have the same parity, all $\overline{a_{d_i}}~(1\le i\le s)$ are odd. In addition, since $\overline{a_{d_i}}=P^{ii}_{min}\in T_X^{\lambda}$ and all entries of $T_X^{\lambda}$ are distinct, these $\overline{a_{d_i}}~(1\le i\le s)$ are naturally different.
	\end{proof}	

In order to prove the main theorem, we first analyze the associated datum of an irreducible homogeneous bundle. Let 
$E_{\lambda}$ be such a bundle, then the associated datum of $E_{\lambda}$ consists of the matrices $P^{ij},~ Q^{ij}~(1\le i\le j\le s)$ and $R^{i}~(1\le i\le s)$. By calculation, we see that
\[
P^{ij}_{uv}=\frac{\sum\limits_{k=d_i-u+1}^{d_j+v-1}\overline{a_k}}{j-i+1}\ge \frac{\sum\limits_{t=i}^{j}\overline{a_{d_t}}}{j-i+1}.\]
Since $\overline{a_{d_1}},\overline{a_{d_2}},\ldots,\overline{a_{d_s}}$ are different odd numbers by Lemma \ref{lemma2}, we have  
\begin{align}\label{pij}
	P^{ij}_{uv}\ge\frac{1+3+\cdots+2(j-i)+1}{j-i+1}=j-i+1.
\end{align}
Similarly, we have
\[
Q^{ij}_{uv}=\frac{\sum\limits_{k=d_{i-1}+u}^{n-1}\overline{a_k}+\sum\limits_{k=d_j+v}^{n-1}\overline{a_k}+2\mathfrak{e}\overline{a_n}}{2s+1-(i+j)}\ge \frac{\sum\limits_{t=i}^{s}\overline{a_{d_t}}+\sum\limits_{t=j+1}^{s}\overline{a_{d_t}}+2\mathfrak{e}\overline{a_n}}{2s+1-(i+j)}.
\]
According to Lemmas \ref{lemma1} and \ref{lemma2}, we know $\mathfrak{e}\overline{a_n}\ge s(s-1)$ and $\overline{a_{d_1}},\overline{a_{d_2}},\ldots,\overline{a_{d_s}}$ are different odd numbers, hence 
$$Q^{ij}_{uv}\ge\frac{(s-i+1)^2+(s-j)^2+2s(s-1)}{2s+1-(i+j)}=\frac{2(s-j)^2+2s(s-1)}{2s+1-(i+j)}+j-i+1.$$
 We can further estimate the value of $Q^{ij}_{uv}$ in two cases.\\
If $j>i$, then 
\begin{align}\label{j>i}
Q^{ij}_{uv}\ge\frac{2(s-j)^2+2s(s-1)}{2s+1-(i+j)}+j-i+1\ge \frac{2s(s-1)}{2s+1-(1+2)}+2=s+2.
\end{align}
If $j=i$, then 
\begin{align}\label{j=i}
Q^{ii}_{uv}\ge\frac{2(s-i)^2+2s(s-1)}{2s+1-2i}+1\ge \frac{2s(s-1)}{2s-1}+1>s.
\end{align}
Finally, in a similar way, we note that
\begin{align}\label{R}
\begin{split}
R^{i}_{uv}&=\frac{\sum\limits_{k=d_{i-1}+u}^{n-1}\overline{a_k}+\sum\limits_{k=d_{i-1}+v}^{n-1}\overline{a_k}+2\mathfrak{e}\overline{a_n}}{2(s+1-i)}\ge\frac{2\sum\limits_{t=i}^{s}\overline{a_{d_t}}+2\mathfrak{e}\overline{a_n}}{2(s+1-i)}\\
&\ge\frac{2(s+1-i)^2+2s(s-1)}{2(s+1-i)}=\frac{s(s-1)}{s+1-i}+s+1-i\\
&\ge 2\sqrt{s(s-1)}.
\end{split}
\end{align}\\
\textbf{Proof of Theorem \ref{mainthm} for type $B_n$ or $C_n$ and $s\ge 3$, $d_s\ne n$:} Suppose there is an irreducible homogeneous Ulrich bundle $E_{\lambda}$ on $X$. Denote $T_X^{\lambda}=\{P^{ij}, Q^{ij}, R^{i}\}$ by the associated datum of $E_{\lambda}$. \\

\textbf{Step 1:} We claim that every $Q^{ij}_{uv}>4$ and $R^{i}_{uv}>4$.

 If $j>i$, then $Q^{ij}_{uv}\ge 5$ by the sequence (\ref{j>i}) and if $j=i$ and $s\ge 4$, then $Q^{ij}_{uv}>4$ by (\ref{j=i}). For the case $s=3$ and $j=i$, by (\ref{j=i}) we know $Q^{ii}_{uv}\ge\frac{2(3-i)^2+12}{7-2i}+1$. It is easy to see that the latter is always greater than $4$ when $1\le i\le 3$. Hence we conclude $Q^{ij}_{uv}>4$. At the same time, based on the sequence (\ref{R}), we can easily conclude that every $R^{i}_{uv}$ is greater than $4$. Therefore, $2$ and $4$ can only be entries of the matrices $P^{ij}$. \\

\textbf{Step 2:} Determine the probability of $2$ in the matrices $P^{ij}$.

By Lemmas \ref{lemma1} and \ref{lemma2}, we know $\overline{a_{d_i}}$ is odd for any $1\le i\le s$ and $\overline{a_k}$ is even for any $k$ except $d_1,\ldots,d_s$. Hence when $j-i$ is even, the sum $\sum_{k=d_i-u+1}^{d_j+v-1}\overline{a_k}$ is odd for any $1\le u\le d_i-d_{i-1}$ and $1\le v\le d_{j+1}-d_j$. Since $E_{\lambda}$ is an Ulrich bundle, every entry $P^{ij}_{uv}$ is an integer. Notice that the numerator of $P^{ij}_{uv}$ is $\sum_{k=d_i-u+1}^{d_j+v-1}\overline{a_k}$,which is odd when $j-i$ is even. Thus every $P^{ij}_{uv}$ is also odd when $j-i$ is even. It follows that $2$ and $4$ can only appear in some matrices $P^{ij}$, where $j-i$ is odd. On the other hand, by (\ref{pij}), we have $P^{ij}_{uv}\ge j-i+1>2$ for any $j>i+1$. Therefore $2$ can only be the smallest entry of $P^{t,t+1}$ for some integer $t~(1\le t\le s-1)$, i.e.,

\begin{align}\label{2}
2=P^{t,t+1}_{min}=P^{t,t+1}_{11}=\frac{\sum\limits_{k=d_t}^{d_{t+1}}\overline{a_k}}{2}.
\end{align}
  Since for any $d_t+1\le k\le d_{t+1}-1$, $\overline{a_k}\ge 2$ by Lemma \ref{lemma1}, we conclude that $d_{t+1}=d_t+1$. And in this case $2=\frac{\overline{a_{d_t}}+\overline{a_{d_{t+1}}}}{2}$, which means that either $\overline{a_{d_t}}=1$ and $\overline{a_{d_{t+1}}}=3$ or $\overline{a_{d_t}}=3$ and $\overline{a_{d_{t+1}}}=1$. Similarly, from the sequence (\ref{pij}) and the parity of $P^{ij}_{uv}$, we conclude that 

 \begin{align}\label{notin}
 4\notin P^{ij},~\text{if}~j-i\ge 4~\text{or}~j-i~\text{is even}.
 \end{align} 

In order to prove the nonexistence of irreducible homogeneous Ulrich bundles on $X$, in the next step we try to judge that $4$ does not appear in the associated datum $T_X^{\lambda}$ of $E_{\lambda}$.\\

\textbf{Step 3:} We prove $4$ does not appear in $T_X^{\lambda}$.

From the above arguments, we know that there exists an integer $t~(1\le t\le s-1)$ such that $d_{t+1}=d_t+1$ and $(\overline{a_{d_t}},\overline{a_{d_{t+1}}})=(1,3)$ or $(3,1)$. Hence $\overline{a_{d_i}}\ge 5$ for any $i\neq t,t+1$ by Lemma \ref{lemma2}. It follows that if $j\le t-1$ or $i\ge t+2$, then 
\[
P^{ij}_{uv}\ge \frac{\sum\limits_{l=i}^{j}\overline{a_{d_l}}}{j-i+1}\ge\frac{5+7+\cdots+2(j-i)+5}{j-i+1}=5+j-i\ge5.
\] 
And if $(i,j)=(t-3,t)$ or $(i,j)=(t+1,t+4)$, then
\begin{align*}
P^{ij}_{uv}&\ge \frac{\sum\limits_{l=i}^{j}\overline{a_{d_l}}}{j-i+1}\ge\frac{1+5+7+\cdots+2(j-i)+3}{j-i+1}=\frac{1+4(j-i)+(j-i)^2}{j-i+1}\\
&>\frac{1+2(j-i)+(j-i)^2}{j-i+1}=j-i+1=4.
\end{align*}
Combining the above analysis with the statement (\ref{notin}), we see that $4$ can only appear in the matrix $P^{t,t+1}$, $P^{t+1,t+2}$, $P^{t-1,t}$, $P^{t,t+3}$, $P^{t-1,t+2}$ or $P^{t-2,t+1}$. Next, we prove the nonexistence of $4$ for the following six cases.\\

Case (1): Suppose $4\in P^{t,t+1}$.

Note that $P^{t,t+1}_{11}=2$ (see (\ref{2})). Hence $4$ would be equal to $P^{t,t+1}_{12}$ or $P^{t,t+1}_{21}$. By calculation, we find \[
P^{t,t+1}_{12}=2+\frac{\overline{a_{d_{t+1}+1}}}{2}~\text{and}~ P^{t,t+1}_{21}=2+\frac{\overline{a_{d_t-1}}}{2}.
\]
Due to Lemma \ref{lemma1}, we know
$lcm(1,2,\ldots,2s-t-1)\mid\overline{a_{d_{t+1}+1}}$ and $lcm(1,2,\ldots,2s-t+1)\mid\overline{a_{d_t-1}}$. Since $s\ge 3$, $\overline{a_{d_{t+1}+1}}$ and $\overline{a_{d_t-1}}$ are greater than $4$, hence $4$ does not appear in the matrix $P^{t,t+1}$, contrary to hypothesis. \\

Case (2): Suppose $4\in P^{t+1,t+2}$.

In this case, $4$ would be the smallest entry in $P^{t+1,t+2}$, i.e. $$4=P^{t+1,t+2}_{min}=\frac{\sum\limits_{k=d_{t+1}}^{d_{t+2}}\overline{a_k}}{2}.$$ Thus $\sum_{k=d_{t+1}}^{d_{t+2}}\overline{a_k}=8$ and
$P^{t,t+2}_{min}=\frac{8+\overline{a_{d_t}}}{3}$.
Since $E_{\lambda}$ is an Ulrich bundle, $P^{t,t+2}_{min}\in\mathbb{Z}$, which forces that $\overline{a_{d_t}}=1$ and hence $\overline{a_{d_{t+1}}}=3$. Then there would be two identical entries $P^{t,t+2}_{min}$ and $\overline{a_{d_{t+1}}}$ in the associated datum $T_X^{\lambda}$, which contradicts $E_{\lambda}$ being Ulrich. \\

Case (3): Suppose $4\in P^{t-1,t}$.

The proof for this case is similar to Case (2). Using $4=P^{t-1,t}_{min}$ and $P^{t-1,t+1}_{min}\in\mathbb{Z}$, we get 
 $\overline{a_{d_t}}=3$ and $\overline{a_{d_{t+1}}}=1$. It follows that $P^{t-1,t+1}_{min}=3=\overline{a_{d_t}}$, contrary to the hypothesis that $E_{\lambda}$ is an Ulrich bundle.\\
 
 Case (4): Suppose $4\in P^{t,t+3}$.
 
 In this case, $4$ would be the smallest entry in $P^{t,t+3}$, which means that 
 \[
 4=P^{t,t+3}_{min}=\frac{\sum\limits_{k=d_t}^{d_{t+3}}\overline{a_k}}{4}=\frac{4+\sum\limits_{k=d_{t+1}+1}^{d_{t+3}}\overline{a_k}}{4}.\] 
 Since $\overline{a_{d_{t+2}}}$ and $\overline{a_{d_{t+3}}}$ are greater than $5$ and they are different odd numbers, the above equation implies that $d_{t+2}=d_{t+1}+1$, $d_{t+3}=d_{t+2}+1$ and $(\overline{a_{d_{t+2}}}, \overline{a_{d_{t+3}}})=(5,7)$ or $(7,5)$. Since $E_{\lambda}$ is an Ulrich bundle,  $P^{t,t+2}_{min}=\frac{4+\overline{a_{d_{t+2}}}}{3}\in\mathbb{Z}$, which implies that $\overline{a_{d_{t+2}}}=5$. Hence $P^{t,t+2}_{min}=3$, which coincides with $\overline{a_{d_t}}$ or $\overline{a_{d_{t+1}}}$, contrary to the hypothesis that $E_{\lambda}$ is an Ulrich bundle.	\\
 
  Case (5): Suppose $4\in P^{t-1,t+2}$.
 
 The proof for this case is similar to Case (4). Using $4=P^{t-1,t+2}_{min}$, we can get $d_t=d_{t-1}+1$ , $d_{t+2}=d_{t+1}+1$ and $(\overline{a_{d_{t-1}}}, \overline{a_{d_{t+2}}})=(5,7)$ or $(7,5)$. Since $P^{t-1,t+1}_{min}=\frac{4+\overline{a_{d_{t-1}}}}{3}\in\mathbb{Z}$, we have $\overline{a_{d_{t-1}}}=5$ and hence $P^{t-1,t+1}_{min}=3$. This leads to a contradiction, because $P^{t-1,t+1}$ coincides with $\overline{a_{d_t}}$ or $\overline{a_{d_{t+1}}}$.	\\
 
 Case (6): Suppose $4\in P^{t-2,t+1}$.
 
 The proof for this case is similar to Case (5). Using $4=P^{t-2,t+1}_{min}$ and  $P^{t-1,t+1}_{min}=\frac{4+\overline{a_{d_{t-1}}}}{3}\in\mathbb{Z}$, we have $\overline{a_{d_{t-1}}}=5$ and hence $P^{t-1,t+1}_{min}=3$, which leads to a contradiction.
 
In conclusion, through detailed analysis of six different cases, we have excluded one by one the possibility of $4$ appearing in the associated datum $T_X^{\lambda}$. 
  Therefore, if $s\ge 3$, there are no irreducible homogeneous Ulrich bundles on $X$.\\

Next, we show that the same statement holds for any isotropic flag variety $B_n/P_{d_1,d_2}$ or $C_n/P_{d_1,d_2}~(d_2\ne n)$, whose Picard number is $2$.\\

\textbf{Proof of Theorem \ref{mainthm} for type $B_n$ or $C_n$ and $s=2$, $d_s\ne n$:}		
 Suppose that $E_{\lambda}$ is an	irreducible homogeneous Ulrich bundle on $X:=B_n/P_{d_1,d_2}$ or $C_n/P_{d_1,d_2}~(d_2\ne n)$. Denote $T_X^{\lambda}=\{P^{ij}, Q^{ij}, R^{i}\}$ by the associated datum of $E_{\lambda}$. According to Theorem \ref{FN},  $T_X^{\lambda}=\{1,2,\ldots, \dim X\}$. Hence $2$ would appear in $T_X^{\lambda}$. By the hypothesis $s=2$, the sequence (\ref{j>i}) and (\ref{j=i}), we get every $Q^{ij}_{uv}$ is always greater than $2$. And by the sequence (\ref{R}), we get $R^{i}_{uv}$ is also greater than $2$. Hence $2$ could only be an entry of some matrix $P^{ij}$, where $1\le i\le j\le 2$.

As every entry of the matrix $P^{11}$ or $P^{22}$ is odd, $2$ appears as an entry of $P^{12}$. By the sequence (\ref{pij}), we get $P^{12}_{uv}$ is at least $2$ and hence $2$ should be the smallest entry in the matrix $P^{12}$. It follows that $2=P^{12}_{min}=\frac{\sum_{k=d_1}^{d_2}\overline{a_k}}{2}$. Since for any $d_1+1\le k\le d_2-1$, $\overline{a_k}\ge 2$ by Lemma \ref{lemma1}, we conclude that 
\begin{align}\label{d_2=d_1+1}
d_2=d_1+1~ \text{and}~(\overline{a_{d_1}}, \overline{a_{d_2}})=(1,3) ~ \text{or}~ (3,1).
\end{align}

On the other hand, by Lemmas \ref{lemma1} and \ref{lemma2}, it is not hard to see all entries in the matrices $P^{11}$, $P^{22}$, $Q^{11}$ and $Q^{22}$ are odd. In addition, since $d_2-d_1=1$, the matrix $R^2$ has only one entry $R^2_{11}=\sum_{k=d_2}^{n-1}\overline{a_k}+\mathfrak{e}\overline{a_n}$, which is odd due to Lemmas \ref{lemma1} and \ref{lemma2}. Thus even numbers can only appear in the matrix $P^{12}$, $Q^{12}$ or $R^1$. According to \cite[Page 10 (3.7)]{Fang}, we know $\dim X=\max\{m_1,m_2\}$, where  $m_1:=\sum_{k=1}^{d_1}\overline{a_k}$ and $m_2:=\sum_{k=d_2}^{n-1}\overline{a_k}+\sum_{k=d_2+1}^{n-1}\overline{a_k}+2\mathfrak{e}\overline{a_n}$. Since $m_1$ and $m_2$ are odd, $\dim X$ is odd. Hence $\dim X-1$ is even and is the largest entry among the matrices $P^{12}$, $Q^{12}$ and $R^1$. By comparing the largest entries of the three matrices, we find 
\[
Q^{12}_{max}=\frac{\sum\limits_{k=1}^{n-1}\overline{a_k}+\sum\limits_{k=d_2+1}^{n-1}\overline{a_k}+2\mathfrak{e}\overline{a_n}}{2}>R^{1}_{max}=\frac{\sum\limits_{k=1}^{n-1}\overline{a_k}}{2}+\frac{\mathfrak{e}\overline{a_n}}{2}
>P^{12}_{max}=\frac{\sum\limits_{k=1}^{n-1}\overline{a_k}}{2}.
\]
Therefore, $\dim X-1=Q^{12}_{max}$. It follows that
\[
R^1_{1,d_1}=\frac{\sum\limits_{k=1}^{n-1}\overline{a_k}+\sum\limits_{k=d_1}^{n-1}\overline{a_k}+2\mathfrak{e}\overline{a_n}}{4}=\frac{2\dim X-2+\overline{a_{d_1}}+\overline{a_{d_2}}}{4}=\frac{\dim X+1}{2}.
\]
As $\dim X$ is the maximum of $m_1$ and $m_2$ and $\dim X-1$ is their average, we have 
\begin{align*}
(m_1,m_2)=(\dim X, \dim X-2)~\text{or}~(\dim X-2,\dim X).
\end{align*} 
Combing the above sequence with (\ref{d_2=d_1+1}), it is not hard to see that either $m_2+\overline{a_{d_1}}$ or $m_2+\overline{a_{d_2}}$ must be equal to $\dim X+1$. If $m_2+\overline{a_{d_1}}=\dim X+1$, then 

\[
Q^{12}_{d_1,1}=\frac{\sum\limits_{k=d_1}^{n-1}\overline{a_k}+\sum\limits_{k=d_2+1}^{n-1}\overline{a_k}+2\mathfrak{e}\overline{a_n}}{2}=\frac{m_2+\overline{a_{d_1}}}{2}=\frac{\dim X+1}{2}=R^1_{1,d_1}.
\]
If $m_2+\overline{a_{d_2}}=\dim X+1$, then
\[
R^2_{11}=\sum_{k=d_2}^{n-1}\overline{a_k}+\mathfrak{e}\overline{a_n}=\frac{m_2+\overline{a_{d_2}}}{2}=\frac{\dim X+1}{2}=R^1_{1,d_1}.
\]
Therefore, in any case, there would always be two identical numbers in $T_X^{\lambda}$, which is contradictory to $E_{\lambda}$ being Ulrich. Therefore, we are done.

\subsubsection{$J=\{d_1,\ldots,d_s\}$ with $d_s=n$}
 According to \cite[Section 3.2.1 II Case (b)]{Fang}, the associated datum $T_X^{\lambda}$ of $E_{\lambda}$ is of the following form.\\

$T_X^{\lambda}=\{\tilde{P}^{ij}_{uv}, \tilde{Q}^{ij}_{uv}~(1\le i\le j\le s-1),\tilde{R}^{i}_{uv}~(1\le i\le s)\}$, where
\begin{align*}&\tilde{P}^{ij}_{uv}=\frac{\sum\limits_{k=d_i-u+1}^{d_j+v-1}\overline{a_k}}{j-i+1}~(1\le u\le d_i-d_{i-1},1\le v\le d_{j+1}-d_j);\\&\tilde{Q}^{ij}_{uv}=\frac{\sum\limits_{k=d_{i-1}+u}^{n-1}\overline{a_k}+\sum\limits_{k=d_j+v}^{n-1}\overline{a_k}+2\mathfrak{e}\overline{a_n}}{2(s+\mathfrak{e})-(i+j+1)}~(1\le u\le d_i-d_{i-1},1\le v\le d_{j+1}-d_j);\\&\tilde{R}^{i}_{uv}=\frac{\sum\limits_{k=d_{i-1}+u}^{n-1}\overline{a_k}+\sum\limits_{k=d_{i-1}+v}^{n-1}\overline{a_k}+2\mathfrak{e}\overline{a_n}}{2(s+\mathfrak{e}-i)}~(1\le u\le v\le d_i-d_{i-1}).  
\end{align*}
It should be noted that the matrices $\tilde{P}^{ij}$ and $\tilde{Q}^{ij}$ can only happen when $s\ge 2$.

By observing, it is not difficult to see that 
\[
\tilde{R}^{i}_{min}=\frac{2\sum\limits_{k=d_i}^{n-1}\overline{a_k}+2\mathfrak{e}\overline{a_n}}{2(s+\mathfrak{e}-i)}~\text{and}~\tilde{Q}^{ij}_{min}=\frac{\sum\limits_{k=d_i}^{n-1}\overline{a_k}+\sum\limits_{k=d_{j+1}}^{n-1}\overline{a_k}+2\mathfrak{e}\overline{a_n}}{2(s+\mathfrak{e})-(i+j+1)}.
\]
By calculation, we have
\begin{align*}
\tilde{Q}^{ij}_{min}&=\frac{\sum\limits_{k=d_i}^{n-1}\overline{a_k}+\sum\limits_{k=d_{j+1}}^{n-1}\overline{a_k}+2\mathfrak{e}\overline{a_n}}{2(s+\mathfrak{e})-(i+j+1)}=\frac{2\sum\limits_{k=d_i}^{n-1}\overline{a_k}+2\mathfrak{e}\overline{a_n}+2\sum\limits_{k=d_{j+1}}^{n-1}\overline{a_k}+2\mathfrak{e}\overline{a_n}}{4(s+\mathfrak{e})-2(i+j+1)}\\
&=\frac{2(s+\mathfrak{e}-i)\tilde{R}^{i}_{min}+2(s+\mathfrak{e}-j-1)\tilde{R}^{j+1}_{min}}{4(s+\mathfrak{e})-2(i+j+1)}.
\end{align*}
Hence, we conclude that
\begin{align}\label{min}
\tilde{R}^{i}_{min}\le\tilde{Q}^{ij}_{min}\le\tilde{R}^{j+1}_{min}~\text{or}~\tilde{R}^{j+1}_{min}\le\tilde{Q}^{ij}_{min}\le\tilde{R}^{i}_{min}
\end{align}
for any $1\le i\le j\le s-1$.

Similar to the previous section, we can obtain the following two lemmas, which are crucial for the proof of our main theorem.
\begin{lem}\label{lemma3}
	Let $E_{\lambda}$ be an irreducible homogeneous bundle on $X$ with highest weight $\lambda=a_1\lambda_1+\cdots+a_n\lambda_n$. If $E_{\lambda}$ is an Ulrich bundle, then for any integer $1\le i\le s$ and $d_{i-1}+1\le k\le d_{i}-1$, we have
	$$lcm(1,2,\ldots,s-i,s-i+2\mathfrak{e},s-i+2\mathfrak{e}+1,\ldots,  2s+2\mathfrak{e}-1-i)\mid\overline{a_k}.$$
\end{lem}
\begin{proof}
Let's follow the proof of Lemma \ref{lemma1}. Fix an integer $i~(1\le i\le s-1)$ and let $j$ be an integer between $i$ and $s-1$. Using the differences $\tilde{P}^{ij}_{u+1,1}-\tilde{P}^{ij}_{u,1}$ and $\tilde{Q}^{ij}_{u,1}-\tilde{Q}^{ij}_{u+1,1}$ are integers, we get
	$$j-i+1\mid\overline{a_{d_{i-1}+u}}~\text{and}~2(s+\mathfrak{e})-(i+j+1)\mid\overline{a_{d_{i-1}+u}}.$$
	Since $j$ runs through all integers between $i$ and $s-1$, for any $1\le u\le d_i-d_{i-1}-1$, we have
	\begin{align}\label{lcm5}
		lcm(1,2,\ldots,s-i)\mid\overline{a_{d_{i-1}+u}}~\text{and}~lcm(s-i+2\mathfrak{e},\ldots,2(s+\mathfrak{e}-i)-1)\mid\overline{a_{d_{i-1}+u}}.
	\end{align}
Moreover, since $
\tilde{R}^{i}_{u,d_i-d_{i-1}}-\tilde{R}^{i}_{u+1,d_i-d_{i-1}}\in\mathbb{Z}$, we have \begin{align}\label{lcm6}
	2(s+\mathfrak{e}-i)\mid\overline{a_{d_{i-1}+u}}.
\end{align}
	Furthermore, when $i\ge 2$, using the difference $\tilde{Q}^{i',i-1}_{1,u}-\tilde{Q}^{i',i-1}_{1,u+1}$ is an integer, where $1\le i'\le i-1$, we have
	 $2(s+\mathfrak{e})-(i'+i)\mid\overline{a_{d_{i-1}+u}}$. Since $i'$ runs through all integers between $1$ and $i-1$, we get $lcm(2(s+\mathfrak{e}-i)+1,\ldots, 2s+2\mathfrak{e}-1-i)\mid\overline{a_{d_{i-1}+u}}$. This together with (\ref{lcm5}) and (\ref{lcm6}) gives us 
	\begin{align}\label{lcm7}
		lcm(1,2,\ldots,s-i,s-i+2\mathfrak{e},s-i+2\mathfrak{e}+1,\ldots,  2s+2\mathfrak{e}-1-i)\mid\overline{a_k}
	\end{align}
	for any integer $1\le i\le s-1$ and $d_{i-1}+1\le k\le d_{i}-1$.
	
Similarly, using the differences $
\tilde{R}^{s}_{1,v}-\tilde{R}^{s}_{1,v+1}$ and $\tilde{Q}^{i,s-1}_{1,v}-\tilde{Q}^{i,s-1}_{1,v+1}$ are integers, we obtain that $2\mathfrak{e}\mid\overline{a_{d_{s-1}+v}}$ and
	$s+2\mathfrak{e}-i\mid\overline{a_{d_{s-1}+v}}$ for any integer $1\le v\le n-d_{s-1}-1$. Since $i$ can run through all integers between $1$ and $s-1$, we have
	\begin{align}\label{lcm8}
		lcm(2\mathfrak{e},2\mathfrak{e}+1,\ldots,s+2\mathfrak{e}-1)\mid\overline{a_k}
	\end{align}
	for any integer $d_{s-1}+1\le k\le n-1$. By combining sequences (\ref{lcm7}) and (\ref{lcm8}), we complete the proof.
\end{proof}

\begin{lem}\label{lemma4}
		With the notations as above, if $E_{\lambda}$ is an Ulrich bundle, then $\overline{a_{d_1}},\ldots,\overline{a_{d_{s-1}}},\overline{a_{d_s}}(=\overline{a_n})$ are $s$ different odd numbers.
\end{lem}
\begin{proof}
By Lemma \ref{lemma3}, we know that $\overline{a_k}$ is even for any $d_i+1\le k\le d_{i+1}-1$. Using similar arguments in Lemma \ref{lemma2}, we can easily get that all integers $\overline{a_{d_i}}~(1\le i\le s-1)$ have the same parity. Next, let us consider the difference
\[
\tilde{Q}^{s-1,s-1}_{min}-\tilde{R}^{s-1}_{min}=\frac{\sum_{k=d_{s-1}}^{n-1}\overline{a_k}+2\mathfrak{e}\overline{a_n}}{2\mathfrak{e}+1}-\frac{2\sum_{k=d_{s-1}}^{n-1}\overline{a_k}+2\mathfrak{e}\overline{a_n}}{2(1+\mathfrak{e})}=\frac{2\mathfrak{e}(\overline{a_n}-\sum_{k=d_{s-1}}^{n-1}\overline{a_k})}{2(1+\mathfrak{e})(2\mathfrak{e}+1)}.
\]
Because $E_{\lambda}$ is an Ulrich bundle, the above difference should be an integer. Note that $\frac{2\mathfrak{e}}{2(1+\mathfrak{e})(2\mathfrak{e}+1)}$ is $\frac{1}{6}$, whether $\mathfrak{e}=\frac{1}{2}$ or $1$. So $\overline{a_n}-\sum_{k=d_{s-1}}^{n-1}\overline{a_k}$ is even, which implies that $\overline{a_n}(=\overline{a_{d_s}})$ and $\overline{a_{d_{s-1}}}$ have the same parity. Therefore, all $\overline{a_{d_i}}~(1\le i\le s)$ have the same parity. Since $\min_{1\le i\le s}\{\overline{a_{d_i}}\}=1$, every $\overline{a_{d_i}}$ is odd. In addition, since for any $1\le i\le s-1$, $\overline{a_{d_i}}=\tilde{P}^{ii}_{min}$ and $\overline{a_n}=\tilde{R}^{s}_{min}$, they are naturally distinct.
\end{proof}	

Now, let's prove the nonexistence of irreducible homogeneous Ulrich bundles on $X$.\\ 

\textbf{Proof of Theorem \ref{mainthm} for type $B_n$ or $C_n$ and $d_s=n$:}		
Suppose there is an	irreducible homogeneous Ulrich bundle $E_{\lambda}$ on $X$ with $\lambda=a_1\lambda_1+\cdots+a_n\lambda_n$. Denote $T_X^{\lambda}=\{\tilde{P}^{ij}, \tilde{Q}^{ij}, \tilde{R}^{i}\}$ by the associated datum of $E_{\lambda}$. Then $2$ would appear in $T_X^{\lambda}$. By Lemmas \ref{lemma3} and \ref{lemma4}, we find that every entry of $\tilde{P}^{ii}~(1\le i\le s-1)$ is odd. In addition, if $n-d_{s-1}=1$, then $\tilde{R}^{s}$ has only one entry $\overline{a_n}$. If $n-d_{s-1}>1$, then the second smallest entry of $\tilde{R}^{s}$ is $\frac{\overline{a_{n-1}}+2\mathfrak{e}\overline{a_n}}{2\mathfrak{e}}$, which is greater than or equal to $3$, as $\overline{a_{n-1}}\ge 2\mathfrak{e}(1+2\mathfrak{e})$ by the hypothesis $s\ge 2$ and Lemma \ref{lemma3}. Hence $2$ can only appear as the smallest entry of some $\tilde{P}^{ij}~(1\le i<j\le s-1)$, $\tilde{R}^{i}~(1\le i\le s-1)$ or $\tilde{Q}^{ij}$. We first claim $2\notin \tilde{Q}^{ij}$.\\

\textbf{Claim:} $2$ does not appear in $\tilde{Q}^{ij}$.

Suppose there is a pair $(i_0,j_0)~(1\le i_0\le j_0\le s-1)$ such that $2=\tilde{Q}^{i_0,j_0}_{min}$, then by the sequence (\ref{min}), either  $\tilde{R}^{i_0}_{min}$ or  $\tilde{R}^{j_0+1}_{min}$ must be equal to $1$. On the other hand, since $E_{\lambda}$ is an Ulrich bundle, we have $\min\{\overline{a_{d_1}},\ldots,\overline{a_{d_{s-1}}},\overline{a_n}\}=1$. Hence $\overline{a_n}=\tilde{R}^{s}_{min}=1$, which implies that $j_0=s-1$. From the equality 
$$2=\tilde{Q}^{i_0,s-1}_{min}=\frac{(s+\mathfrak{e}-i)\tilde{R}^{i_0}_{min}+\mathfrak{e}}{s+2\mathfrak{e}-i},$$
we get $\tilde{R}^{i_0}_{min}=2+\frac{\mathfrak{e}}{s+\mathfrak{e}-i}$. It would not be an integer, whether $\mathfrak{e}=\frac{1}{2}$ or $1$, contrary to the hypothesis $E_{\lambda}$ is an Ulrich bundle. Therefore, $2$ does not appear in $\tilde{Q}^{ij}$.\\

Next, we prove that $2$ does not appear in either $\tilde{R}^{i}$ or $\tilde{P}^{ij}$. Assume there is some $i_0~(1\le i_0\le s-1)$ such that  $2=\tilde{R}^{i_0}_{min}=\frac{2\sum_{k=d_{i_0}}^{n-1}\overline{a_k}+2\mathfrak{e}\overline{a_n}}{2(s+\mathfrak{e}-i_0)}$. Then we have $4(s+\mathfrak{e}-i_0)=2(\sum_{k=d_{i_0}}^{n-1}\overline{a_k}+\mathfrak{e}\overline{a_n})$. Note that the left-hand side of the above equation is an even number and $\overline{a_n}$ is odd by Lemma \ref{lemma4}, we must have $\mathfrak{e}=1$, i.e., $2$ appears in $\tilde{R}^{i_0}$ only if $G$ is of type $C$. By substituting $\mathfrak{e}=1$ into the above equation, we get 
\begin{align}\label{fff}
4(s+1-i_0)=2\sum_{k=d_{i_0}}^{n}\overline{a_k}\ge 2(s-i_0+1)^2.
\end{align}
 Note that for the last inequality, we use the assertion that $\overline{a_{d_{i_0}}},\ldots, \overline{a_{d_{s-1}}}, \overline{a_{d_{s}}}(=\overline{a_n})$ are different odd numbers (see Lemma \ref{lemma4}). From (\ref{fff}), we can deduce that $i_0=s-1$. Hence $\sum_{k=d_{s-1}}^{n}\overline{a_k}=4$, which implies that $d_{s-1}=n-1$ and $\overline{a_n}$ is equal to $1$ or $3$. In either case, $\tilde{Q}^{s-1,s-1}_{min}=\frac{\sum_{k=d_{s-1}}^{n-1}\overline{a_k}+2\overline{a_n}}{3}=\frac{4+\overline{a_n}}{3}$ would not be an integer, which contradicts the hypothesis.\\
 
Assume $2$ appears in some $\tilde{P}^{ij}~(1\le i<j\le s-1)$, then $s\ge 3$ and
$2$ is equal to 
 $\tilde{P}^{t_0,t_0+1}_{min}$ for some $1\le t_0\le s-2$ (the reason is the same as Step $2$ in the proof for the case $s\ge 3$ and $d_s\ne n$ in Section \ref{section1}).
Then using Lemmas \ref{lemma3} and \ref{lemma4}, we can infer that $d_{t_0+1}=d_{t_0}+1$ and $(\overline{a_{d_{t_0}}},\overline{a_{d_{t_0+1}}})=(1,3)$ or $(3,1)$. Since $s\ge 3$, appling almost verbatim Step $3$ in Section \ref{section1}, we get $4$ would not appear as an entry of $\tilde{P}^{ij}$. Hence $4$ can only be equal to $\tilde{R}^{j_0}_{min}$ for some $1\le j_0\le s-1$ by (\ref{min}). Similar to the previous arguments, we would get $\mathfrak{e}=1$ and $j_0$ satisfies $s-3\le j_0\le s-1$. It is easy to judge that $j_0\ne s-2$, otherwise $4(s+1-j_0)=\sum_{k=d_{s-2}}^{n}\overline{a_k}$, which is absurd, because the left-hand side of the equality is even, but the right-hand side is odd due to Lemmas \ref{lemma3} and \ref{lemma4}. 

If $j_0=s-3$, then $16=\sum_{k=d_{s-3}}^{n}\overline{a_k}$. As $\sum_{k=d_{s-3}}^{n-1}\overline{a_k}\ge 9$, $\overline{a_n}\ge 5$ and $\overline{a_n}$ is odd, we have $\overline{a_n}=5$ or $7$. Then $\tilde{Q}^{s-3,s-1}_{min}=\frac{\sum_{k=d_{s-3}}^{n-1}\overline{a_k}+2\overline{a_n}}{5}=\frac{16+\overline{a_n}}{5}$ would not be an integer, whether $\overline{a_n}=5$ or $7$.

If $j_0=s-1$, then $8=\sum_{k=d_{s-1}}^{n}\overline{a_k}$. Since $\overline{a_n}\ge 5$ and it is odd, $\sum_{k=d_{s-1}}^{n-1}=3$, $\overline{a_n}=5$ or $\sum_{k=d_{s-1}}^{n-1}=1$, $\overline{a_n}=7$. If the former happens, $\tilde{Q}^{s-1,s-1}_{min}=\frac{\sum_{k=d_{s-1}}^{n-1}\overline{a_k}+2\overline{a_n}}{3}=\frac{13}{3}\notin \mathbb{Z}$. If the latter happens, the equality $\sum_{k=d_{s-1}}^{n-1}\overline{a_k}=1$ implies $d_{s-1}=n-1$ and $\overline{a_{d_{s-1}}}=1$. Thus we get $t_0=s-2$ and $\overline{a_{d_{s-2}}}=3$. It leads to $\tilde{R}^{s-2}_{min}=\frac{\sum_{k=d_{s-2}}^{n}\overline{a_k}}{3}=\frac{3+1+7}{3}=\frac{11}{3}\notin \mathbb{Z}$, which contrary to the hypothesis $E_{\lambda}$ is an Ulrich bundle.
 
 In summary, we conclude that there are no irreducible homogeneous Ulrich bundles on $X$ with respect to the minimal ample class.\\	

\subsection{Type $D_n$}
When $G$ is of type $D_n$, the existence problem gets more intricate due to the unique structure of its Dynkin diagram. We need to use different methods to consider the existence of Ulrich bundles for the following three cases.
\subsubsection{$J=\{d_1,\ldots,d_s\}$ with $d_s\le n-2$}
 According to \cite[Section 3.2.1 III Case (a)]{Fang}, the associated datum $T_X^{\lambda}$ of $E_{\lambda}$ is of the following form. (By convention we set $d_0=0$ and $d_{s+1}=n$.)

$T_X^{\lambda}=\{P^{ij}, Q^{ij}~(1\le i\le j\le s),R^{i}~(1\le i\le s)\}$, where
\begin{align*}
	&P^{ij}_{uv}=\frac{\sum\limits_{k=d_i-u+1}^{d_j+v-1}\overline{a_k}}{j-i+1}~(1\le u\le d_i-d_{i-1},1\le v\le d_{j+1}-d_j);\\
	&Q^{ij}_{uv}=\frac{\sum\limits_{k=d_{i-1}+u}^{n-2}\overline{a_k}+\sum\limits_{k=d_j+v}^{n}\overline{a_k}}
	{2s+1-(i+j)}
	~(1\le u\le d_i-d_{i-1},1\le v\le d_{j+1}-d_j);\\
	&R^{i}_{uv}=\frac{\sum\limits_{k=d_{i-1}+u}^{n-2}\overline{a_k}+\sum\limits_{k=d_{i-1}+v}^{n}\overline{a_k}}
	{2(s+1-i)}~(1\le u< v\le d_i-d_{i-1}).
\end{align*}

Similar to the proof of Lemma \ref{lemma1}, we can get the following lemma.
\begin{lem}\label{lcm1D}
	Let $E_{\lambda}$ be an irreducible homogeneous Ulrich bundle on $X$ with highest weight $\lambda=a_1\lambda_1+\cdots+a_n\lambda_n$. 
	\begin{itemize}
		\item[(i)] For any $1\le k\le d_1-1$, we have
		$lcm(1,2,\ldots, 2s-1)\mid\overline{a_k}$.
		\item[(ii)] For any integer $2\le i\le s+1$ and $d_{i-1}+1\le k\le d_{i}-1$, we have
		$$lcm(1,2,\ldots, 2s-2i+1)\mid\overline{a_k}~\text{and}~lcm(2s-2i+3,2s-2i+4,\ldots, 2s-i+1)\mid\overline{a_k}.$$
		\item[(iii)] $\overline{a_n}\ne \overline{a_{n-1}}$ and $lcm(1,2,\ldots,s)\mid\overline{a_n}$.
	\end{itemize}
\end{lem}
\begin{proof}
The proof of (i) and (ii) is almost the same as the proof of Lemma \ref{lemma1}, except that we cannot figure out $\overline{a_k}$ is divisible by $2(s-i+1)$. This is because when $d_i-d_{i-1}$ is equal to $2$, there is only one entry in the matrix $R^{i}$.
	
For the statement (iii), let us compare the difference 
$$Q^{is}_{d_i-d_{i-1},n-d_s}-P^{is}_{1,n-d_s}=\frac{\overline{a_n}-\overline{a_{n-1}}}{s+1-i}.$$
Since $E_{\lambda}$ is an Ulrich bundle, the above difference should be a non-zero integer and hence $(s+1-i)\mid\overline{a_n}-\overline{a_{n-1}}$ for any $1\le i\le s$. By the statement (ii), we know that $lcm(1,2,\ldots,s)\mid\overline{a_{n-1}}$. Therefore $lcm(1,2,\ldots,s)\mid\overline{a_n}$.
\end{proof}

Using Lemma \ref{lcm1D}, we can immediately get the following lemma.
\begin{lem}\label{lcm2D}
	With the same assumption as above, then $\overline{a_{d_1}},\overline{a_{d_2}},\ldots,\overline{a_{d_s}}$ are $s$ different odd numbers.
\end{lem}

With the above lemmas, let us first consider the simple case, where the Picard number of $X$ is greater than $2$.\\

\textbf{Proof of Theorem \ref{mainthm} for type $D_n$ and $s\ge 3$, $d_s\le n-2$:}		
Suppose $E_{\lambda}$ is an Ulrich bundle on $X$.	
Note that 
$$Q^{ij}_{uv}=\frac{\sum\limits_{k=d_{i-1}+u}^{n-2}\overline{a_k}+\sum\limits_{k=d_j+v}^{n-2}\overline{a_k}+\overline{a_{n-1}}+\overline{a_n}}
{2s+1-(i+j)}
~\text{and}~R^{i}_{uv}=\frac{\sum\limits_{k=d_{i-1}+u}^{n-2}\overline{a_k}+\sum\limits_{k=d_{i-1}+v}^{n-2}\overline{a_k}+\overline{a_{n-1}}+\overline{a_n}}{2(s+1-i)}.
$$
By Lemma \ref{lcm1D}, we have known that $\overline{a_{n-1}}\ge s(s-1)$ and $\overline{a_n}\ge s(s-1)$, so $\overline{a_{n-1}}+\overline{a_n}\ge 2s(s-1)$. Since $d_s\le n-2$, as in the previous analysis of Section \ref{section1} Step $1$, it is not difficult to obtain that $Q^{ij}_{uv}$ and $R^{i}_{uv}$ are greater than $4$ when $s\ge 3$. 

Note also that the matrix $P^{ij}$ in the associated datum of type $D$ coincides with the matrix $P^{ij}$ in the associated datum of type $B$, so verbatim following the proof after Step $2$ of Section \ref{section1}, we can conclude that $4$ does not appear in the associated datum of $E_{\lambda}$. It follows that there are no irreducible homogeneous Ulrich bundles on $X$.\\

Next, we deal with the complex case where the Picard number of $X$ is $2$, that is $s=2$. To this end, we need some preparatory work.

\begin{prop}\label{prop1}
	If there exists an irreducible homogeneous Ulrich bundle on $X$, then $d_1\ge 2$ and $d_2=d_1+1$.
\end{prop}
\begin{proof}
Suppose that $E_{\lambda}$ is an Ulrich bundle on $X$. Since $d_2\le n-2$ and $\overline{a_{n-1}}+\overline{a_n}\ge 2s(s-1)=4$, we have every $Q^{ij}_{uv}$ is greater than $2$ and so is $R^{i}_{uv}$.
Hence $2$ could only be an entry of some matrix $P^{ij}$, where $1\le i\le j\le 2$.
As every entry of the matrix $P^{11}$ or $P^{22}$ is odd, $2$ appear as the smallest entry of $P^{12}$,
 i.e., $2=P^{12}_{min}=\frac{\sum_{k=d_1}^{d_2}\overline{a_k}}{2}$, which implies that $	d_2=d_1+1$ and $(\overline{a_{d_1}}, \overline{a_{d_2}})=(1,3)$ or $(3,1)$, as $\overline{a_k}\ge 3$ for any $d_1+1\le k\le d_2-1$ by Lemma \ref{lcm1D} (ii).

Note that $\dim X=|T_X^{\lambda}|=2(d_1+1)(n-d_2)+2d_1+\frac{d_1(d_1-1)}{2}$. If $d_1$ were equal to $1$, then $\dim X$ is even. However by the assumption that $E_{\lambda}$ is an Ulrich bundle, $\dim X$ should be equal to $\max\{\sum_{k=1}^{d_1}\overline{a_k},\sum_{k=d_2}^{n-2}\overline{a_k}+\sum_{k=d_2+1}^{n}\overline{a_k}\}$ (see \cite[Page 11 (3.15)]{Fang}), which implies that $\dim X$ is odd. Hence $d_1\ge 2$.	
\end{proof}

\begin{rem}
Since $d_2-d_1=1$, $T_X^{\lambda}$ consists of matrices $P^{ij},~ Q^{ij}(1\le i\le j\le 2)$ and $R^1$, where $R^{1}$ is a strictly upper triangular matrix, $P^{11}$ and $Q^{11}$ are $(d_1\times 1)$-matrices, $P^{22}$ and $Q^{22}$ are $(1\times (n-d_2))$-matrices. Moreover, from Lemmas \ref{lcm1D} and \ref{lcm2D}, we conclude that all entries in the matrices $P^{11}$, $P^{22}$, $Q^{11}$ and $Q^{22}$ are odd, which is very important for the subsequent discussion.
\end{rem}

\begin{prop}\label{prop2}
If there exists an Ulrich bundle $E_{\lambda}$ with $\lambda=a_1\lambda_1+\cdots+a_n\lambda_n$ on $X$, then 
the following holds:
\begin{itemize}
	\item[(1)] $\overline{a_{d_1}}=1$, $\overline{a_{d_2}}=3$, $\overline{a_{d_2+1}}=4$, $\dim X=\sum_{k=1}^{d_1}\overline{a_k}$ and $\dim X-2=\sum_{k=d_2}^{n-2}\overline{a_k}+\sum_{k=d_2+1}^{n}$.
	\item[(2)] $d_1\ge 3$ and  $12\mid\overline{a_k}$ for any $1\le k\le d_1-1$.
	\item[(3)]  $4\mid \overline{a_n}-\overline{a_{n-1}}$.
\end{itemize}		
\end{prop}
\begin{proof}
Similar to the proof for the case $s=2$ and $d_s\ne n$ in Section \ref{section1}, it is not hard to see $\dim X-1=Q^{12}_{max}$, where $$Q^{12}_{max}=\frac{\sum\limits_{k=1}^{n-2}\overline{a_k}+\sum\limits_{k=d_2+1}^{n}\overline{a_k}}{2}=\frac{\sum\limits_{k=1}^{d_1}\overline{a_k}+\sum\limits_{k=d_2}^{n-2}\overline{a_k}+\sum\limits_{k=d_2+1}^{n}\overline{a_k}}{2}.$$ 
It follows that $\dim X$ and $\dim X-2$ are one of  $\sum_{k=1}^{d_1}\overline{a_k}$ and $\sum_{k=d_2}^{n-2}\overline{a_k}+\sum_{k=d_2+1}^{n}\overline{a_k}$ respectively, as $\dim X$ is the maximum value between these two integers. So there are four possibilities for $\overline{a_{d_1}}$, $\overline{a_{d_2}}$, $m_1:=\sum_{k=1}^{d_1}\overline{a_k}$ and $m_2:=\sum_{k=d_2}^{n-2}\overline{a_k}+\sum_{k=d_2+1}^{n}\overline{a_k}$:
\begin{itemize}
	\item[(a)] $m_1=\dim X$, $m_2=\dim X-2$, $\overline{a_{d_1}}=3$ and $\overline{a_{d_2}}=1$;
	\item[(b)] $m_1=\dim X-2$, $m_2=\dim X$, $\overline{a_{d_1}}=1$ and $\overline{a_{d_2}}=3$;
	\item[(c)] $m_1=\dim X-2$, $m_2=\dim X$, $\overline{a_{d_1}}=3$ and $\overline{a_{d_2}}=1$;
	\item[(d)] $m_1=\dim X$, $m_2=\dim X-2$, $\overline{a_{d_1}}=1$ and $\overline{a_{d_2}}=3$.
\end{itemize}
For the cases $(a)$ and $(b)$, since 
$$
Q^{12}_{d_1,1}=\frac{\overline{a_{d_1}}+m_2}{2}=\frac{\dim X+1}{2}~\text{and}~
R^1_{1,d_1}=\frac{m_1+m_2+\overline{a_{d_1}}+\overline{a_{d_2}}}{4}=\frac{\dim X+1}{2},$$
 this contradicts the hypothesis that $E_{\lambda}$ is an Ulrich bundle and hence neither possibility exists. For the cases $(c)$ and $(d)$, we first prove that $4$ can only be equal to $P^{12}_{12}$.
 
 Since all entries in the matrices $P^{11}$, $P^{22}$, $Q^{11}$ and $Q^{22}$ are odd, $4$ can only appear in the matrix $R^1$, $Q^{12}$ and $P^{12}$. Since $1$, $2$ and $3$ appear as $P^{11}_{min}$,  $P^{12}_{min}=P^{12}_{11}$ and $P^{22}_{min}$ respectively, $4$ can only be equal to the smallest entry of the matrix 
 $R^1$ or $Q^{12}$ or the second smallest entry of $P^{12}$, i.e., $4$ is $R^1_{min}$, $Q^{12}_{min}$ or $\min\{P^{12}_{21},P^{12}_{12}\}$.
 By Lemma \ref{lcm1D}, we know that $\overline{a_{d_1-1}}\ge 6$ and $\overline{a_{n-1}}+\overline{a_n}\ge 4$, hence $P^{12}_{21}=\sum_{k=d_1-1}^{d_2}\overline{a_k}\ge5$ and 
  $R^1_{min}=\frac{\sum_{k=d_1-1}^{n-2}\overline{a_k}+\sum_{k=d_1}^{n}\overline{a_k}}{4}\ge \frac{6+2\times 4+4}{4}>4.$  \\
  
 \textbf{Claim:} $4$ is not equal to $Q^{12}_{min}$, where $Q^{12}_{min}=Q^{12}_{d_1,n-d_2}=\frac{4+\sum_{k=d_2+1}^{n-2}\overline{a_k}+\overline{a_n}}{2}$.
  
  Suppose this were proved, then $4$ can only equal to $P^{12}_{12}=\frac{\sum_{k=d_1}^{d_2+1}\overline{a_k}}{2}$, which implies that $\overline{a_{d_2+1}}=4$. Next, we illustrate that Case (c) is impossible. If Case (c) occurs, it is easy to find 
  $$P^{12}_{d_1,1}=\frac{\sum\limits_{k=1}^{d_2}\overline{a_k}}{2}=\frac{m_1+\overline{a_{d_2}}}{2}=\frac{\dim X-1}{2}~\text{and}~Q^{12}_{d_1,2}
  =\frac{\overline{a_{d_1}}+m_2-\overline{a_{d_2+1}}}{2}=\frac{\dim X-1}{2},$$ 
  which contradicts $E_{\lambda}$ being an Ulrich bundle. So, only Case (d) can happen and that gives us the first statement. Note that in this case, $Q^{11}_{d_1,1}=\frac{\sum_{k=d_1}^{n-2}\overline{a_k}+\sum_{k=d_2}^{n}\overline{a_k}}{3}=\frac{m_2+\overline{a_{d_1}}+\overline{a_{d_2}}}{3}=\frac{\dim X+2}{3}$. Since $E_{\lambda}$ is an Ulrich bundle, $Q^{11}_{d_1,1}$ should be an integer, which implies that $3\mid \dim X+2$. 
  If $d_1$ were $2$, then $\dim X=2(d_1+1)(n-d_2)+2d_1+\frac{d_1(d_1-1)}{2}=6n-13$, which contradicts $\dim X+2$ being divisible by $3$. Hence $d_1\ge 3$. It follows that the matrix $R^1$ has at least three elements. By comparing the difference between any two adjacent elements of the first row and last column of $R^1$, we can see that $\overline{a_k}$ is divisible by $4$ for any $1\le k\le d_1-1$, as all differences must be integers. Moreover, by Lemma \ref{lcm1D}, we know $\overline{a_k}$ is divisible by $6$, which brings us to the second statement. In addition, since $Q^{12}_{d_1,n-d_2}-R^1_{d_1-1,d_1}=\frac{\overline{a_n}-\overline{a_{n-1}}-\overline{a_{d_1-1}}}{4}$ is an integer and $\overline{a_{d_1-1}}$ is divisible by $4$, the third statement follows immediately.\\ 
   
It thus remains to prove our claim. If $4$ were equal to $Q^{12}_{min}$, then $\sum_{k=d_2+1}^{n-2}\overline{a_k}+\overline{a_n}=4$. By Lemma \ref{lcm1D} (ii) and (iii), we know $\overline{a_k}~(d_2+1\le k\le n-1)$ and $\overline{a_n}$ are all even. Thus 
 there are two possibilities: 
 \begin{itemize}
 	\item[(1)]  $\overline{a_n}=2$ and $\sum_{k=d_2+1}^{n-2}\overline{a_k}=2$, which implies that $d_2=n-3$ and $\overline{a_{d_2+1}}=2$;
 	\item[(2)]  $\overline{a_n}=4$ and $\sum_{k=d_2+1}^{n-2}\overline{a_k}=0$, which implies that $d_2=n-2$.
 \end{itemize} 
  We say (1) cannot happen, because in this case $P^{12}_{12}=\frac{\sum_{k=d_1}^{d_2+1}\overline{a_k}}{2}=3$ coincides with $\overline{a_{d_1}}$ or $\overline{a_{d_2}}$. Next, we show that (2) is also impossible for either Case (c) or (d). \\
 
 Assume Cases (c) and (2) happen, then 
 $Q^{22}_{min}=\overline{a_{d_2}}+\overline{a_n}=5$. In this case $\overline{a_{n-1}}(=\overline{a_{d_2+1}})$ cannot be equal to $4$ or $6$, otherwise $P^{12}_{12}=\frac{\sum_{k=d_1}^{d_2+1}\overline{a_k}}{2}$ would be equal to $4$ or $5$, which has appear in the associated datum of $E_{\lambda}$. To prove that Cases (c) and (2) cannot happen, it suffices to state that $6$ does not appear in the associated datum of $E_{\lambda}$. If $6$ appears, since $6$ is even, $6$ can only be $R^1_{min}$, $\min\{Q^{12}_{d_1-1,n-d_2},Q^{12}_{d_1,n-d_2-1}\}$ or $\min\{P^{12}_{21},P^{12}_{12}\}$. We separately prove that all possibilities are impossible.
 
 If $6=R^1_{min}=
 \frac{\overline{a_{d_1-1}}+\overline{a_{n-1}}+12}{4}$, then $\overline{a_{d_1-1}}+\overline{a_{n-1}}=12$. By Lemma \ref{lcm1D} (i), $6\mid \overline{a_{d_1-1}}$, so we have $\overline{a_{d_1-1}}=\overline{a_{n-1}}=6$, this contradicts the previous argument that $\overline{a_{n-1}}$ is not equal to $6$. 

 If $6=Q^{12}_{d_1,n-d_2-1}=
 \frac{8+\overline{a_{n-1}}}{2}$, then $\overline{a_{n-1}}=4$, which is also impossible.
 
 If $6$ is equal to $P^{12}_{21}(=\frac{4+\overline{a_{d_1-1}}}{2})$ or  $Q^{12}_{d_1-1,n-d_2}=(\frac{8+\overline{a_{d_1-1}}}{2})$, then $\overline{a_{d_1-1}}=8$ or $4$, which contradicts $\overline{a_{d_1-1}}$ being divisible by $6$. 
 
  If $6=P^{12}_{12}=\frac{4+\overline{a_{d_2+1}}}{2}$, then $\overline{a_{d_2+1}}=8$. 
 However, this leads to  $Q^{11}_{d_1,1}=\frac{\sum_{k=d_1}^{n-2}\overline{a_k}+\sum_{k=d_2}^{n}\overline{a_k}}{3}=\frac{17}{3}$ is not an integer.
So in summary, $6$ does not appear in the associated datum of $E_{\lambda}$.\\
 
Assume Cases (d) and (2) happen, then 
$Q^{22}_{min}=7$. In this case $\overline{a_{n-1}}$ cannot be equal to $2$ or $4$, otherwise $P^{12}_{12}$ would be equal to $3$ or $4$, which has appear in the associated datum of $E_{\lambda}$. To prove that Cases (d) and (2) cannot happen, it suffices to state that $5$ does not appear in the associated datum $T_X^{\lambda}=\{P^{ij},Q^{ij}~(1\le i\le j\le 2), R^1\}$.

Firstly, since $Q^{22}_{min}=7$, $5$ does not appear in $Q^{22}$.

 Secondly, if $5\in R^1$ or $5\in Q^{11}$, then $5$ would be equal to $R^1_{min}$ or $Q^{11}_{min}$, which implies that $\overline{a_{n-1}}=2$ or $\overline{a_{n-1}}=4$, which is  impossible by the previous argument.

Next, if $5\in Q^{12}$, then $5$ would be equal to $Q^{12}_{d_1-1,n-d_2}$ or $Q^{12}_{d_1,n-d_2-1}$, which implies that $\overline{a_{d_1-1}}=2$ or $\overline{a_{n-1}}=2$. This contradicts Lemma \ref{lcm1D} (i).

 Finally, since $\overline{a_{d_1-1}}\ge 6$ and $\overline{a_{n-1}}\ne 2$, we have $P^{11}_{21}=1+\overline{a_{d_1-1}}\ge 7$ and $P^{22}_{12}=3+\overline{a_{n-1}}\ne 5$. It follows that $5$ does not appear in $P^{11}$ or $P^{22}$. 
If $5\in P^{12}$, then $5$ would be equal to $P^{12}_{12}$ or $P^{12}_{21}$, which implies that $\overline{a_{d_2+1}}=6$ or $\overline{a_{d_1-1}}=6$. However this leads to $Q^{12}_{d_1,1}=7$ or $P^{11}_{21}=7$, which coincides with $Q^{22}_{min}$, so it is also impossible.
In summary, $5$ does not appear in the associated datum of $E_{\lambda}$.
\end{proof}	

Next, we prove there are some restrictions on the associated datum for an arbitrary irreducible homogeneous Ulrich bundle. This is the key to proving that $X$ does not carry irreducible homogeneous Ulrich bundles. Denote $M(i,~)$ by the $i$-th row of the matrix $M$. We write $x\in M(i,~)$ if $x$ is an entry of $M(i,~)$.  
\begin{prop}\label{prop3}
	Let $E_{\lambda}$ be an Ulrich bundle with $\lambda=a_1\lambda_1+\cdots+a_n\lambda_n$ on $X$ and  $T_X^{\lambda}=\{P^{ij}, Q^{ij}~(1\le i\le j\le 2),R^{i}~(1\le i\le 2)\}$ the associated datum of $E_{\lambda}$. Then 
	\begin{align}\label{*}
	\text{for any}~ x\in P^{12}(1,~), \text{we have}~ x+l\in P^{12}(1,~)~\text{or}~ P^{22}(1,~)~\text{for}~ l=1,2\tag{*}.
	\end{align}
\end{prop}
\begin{proof}
We prove the proposition by induction on $x$. By the proof of Proposition \ref{prop1}, we know that the smallest element $P^{12}_{11}$ in $P^{12}(1,~)$ is $2$. From the proof of Proposition \ref{prop2}, we find that $3=\overline{a_{d_2}}\in P^{22}(1,~)$ and $4=P^{12}_{12}\in P^{12}(1,~) $. Therefore, $2$ satisfies the assertion (\ref{*}). Suppose the assertion (\ref{*}) is true for any $y$ with $y<x$ and $y\in P^{12}(1,~)$. Since all entries of the matrices $P^{ii}$ and $Q^{ii}~(i=1,2)$ are odd, at least one of $y+1$ and $y+2$ must lie in $P^{12}(1,~)$ by induction. So we say that the induction hypothesis implies the following two consequences:
\begin{itemize}
	\item[$\normalsize{\textcircled{\scriptsize{1}}}$] For any $y~(y<x)$, we have $y\in P^{12}(1,~)$ or $y\in P^{22}(1,~)$. 
	\item[$\normalsize{\textcircled{\scriptsize{2}}}$] Let $y_1$ and $y_2$ be two adjacent integers in $P^{12}(1,~)$. If $y_1$ and $y_2$ are less than $x$, then $y_1-y_2$ is less than or equal to $2$. 
\end{itemize}
As any integer $x\in P^{12}(1,~)$ corresponds to an integer $2x-1\in P^{22}(1,~)$,  $\normalsize{\textcircled{\scriptsize{2}}}$ is equivalent to say that the difference between any two adjacent integers in $P^{22}(1,~)$ less than $2x-1$ is not more than $4$. Next, we prove the assertion (\ref{*}) is also true for $x\in P^{12}(1,~)$. \\

\textbf{Case 1:} Suppose $x$ is odd. Then $x-1$ is even. According to the assertion $\normalsize{\textcircled{\scriptsize{1}}}$, we have $x-1\in P^{12}(1,~)$ and thus by the induction hypothesis $x+1\in P^{12}(1,~)$ or $P^{22}(1,~)$. On the other hand, by $\normalsize{\textcircled{\scriptsize{1}}}$, we know $x-2\in P^{12}(1,~)$ or $P^{22}(1,~)$. Since $x$ is odd and $x\in P^{12}(1,~)$, we must have $x-2\in P^{22}(1,~)$, otherwise there would be two adjacent odd numbers $a~(a\le x-4)$ and $b~(b\ge x+2)$ in $P^{22}(1,~)$, which contradicts to the assertion $\normalsize{\textcircled{\scriptsize{2}}}$ that $b-a\le 4$. Using the assertion $\normalsize{\textcircled{\scriptsize{2}}}$ again, we find that $x+2\in P^{22}(1,~)$, because $x-2\in P^{22}(1,~)$ and $x\in P^{12}(1,~)$. Thus, if $x$ is odd, the assertion (\ref{*}) is indeed true for $x$.\\

\textbf{Case 2:} Suppose $x$ is even. We can assume that $x\ge 4$. In order to prove the assertion (\ref{*}), it suffices to show that neither $x+1$ nor $x+2$ appear as entries of $Q^{11}$, $Q^{22}$, $R^1$, $P^{11}$, $P^{12}(2,~)$ and $Q^{12}$. Let's prove it in the following steps. 

\textbf{Step 2.1:} We claim that neither $x+1$ nor $x+2$ appear as entries of $Q^{11}$, $Q^{22}$ and $R^1$.

To prove the claim, it suffices to show that the smallest entry of $Q^{11}$, $Q^{22}$ or $R^1$ is greater than $x+2$. 
Since $x\in P^{12}(1,~)$, we can write $x$ as $P^{12}_{1,v}=\frac{\sum_{k=d_1}^{d_2+v-1}\overline{a_k}}{2}$, where $1\le v\le n-d_2$. If $v<n-d_2$, then $\sum_{k=d_1}^{n-2}\overline{a_k}\ge 2x$. By Proposition \ref{prop2}, we get $\overline{a_{d_1-1}}\ge 12$, $\overline{a_{d_1}}=1$, $\overline{a_n}\ge 2$ and $\overline{a_{n-1}}+\overline{a_n}\ge 8$.
Thus, we have
\begin{align*}
&Q^{11}_{min}=\frac{\sum_{k=d_1}^{n-2}\overline{a_k}+\sum_{k=d_2}^{n}\overline{a_k}}{3}=\frac{2\sum_{k=d_1}^{n-2}\overline{a_k}-\overline{a_{d_1}}+\overline{a_{n-1}}+\overline{a_n}}{3}\ge \frac{4x-1+8}{3}>x+3,\\
&Q^{22}_{min}=\sum_{k=d_2}^{n-2}\overline{a_k}+\overline{a_n}=\sum_{k=d_1}^{n-2}\overline{a_k}-\overline{a_{d_1}}+\overline{a_n}\ge 2x-1+2\ge x+5,\\
&R^{1}_{min}=\frac{\sum_{k=d_1-1}^{n-2}\overline{a_k}+\sum_{k=d_1}^{n}\overline{a_k}}{4}=\frac{\overline{a_{d_1-1}}+2\sum_{k=d_1}^{n-2}\overline{a_k}+\overline{a_{n-1}}+\overline{a_n}}{4}\ge \frac{12+4x+8}{4}=x+5.	
\end{align*}
For the first two inequalities, we use the assumption $x\ge 4$. If $v=n-d_2$, then $\sum_{k=d_1}^{n-1}\overline{a_k}=2x$. Similarly, we have 
\begin{align*}
	&Q^{11}_{min}=\frac{2\sum_{k=d_1}^{n-1}\overline{a_k}-\overline{a_{d_1}}+\overline{a_n}-\overline{a_{n-1}}}{3}\ge\frac{4x-1+4}{3}>x+2,\\
	&Q^{22}_{min}=\sum_{k=d_1}^{n-1}\overline{a_k}-\overline{a_{d_1}}+\overline{a_n}-\overline{a_{n-1}}\ge 2x-1+4\ge x+7,\\
	&R^{1}_{min}=\frac{\overline{a_{d_1-1}}+2\sum_{k=d_1}^{n-1}\overline{a_k}+\overline{a_n}-\overline{a_{n-1}}}{4}\ge \frac{12+4x+4}{4}=x+4.	
\end{align*}
Here we use the inequalities $\overline{a_n}-\overline{a_{n-1}}\ge 4$ and $x\ge 4$. \\

\textbf{Step 2.2:} We claim that neither $x+1$ nor $x+2$ appear as entries of $P^{11}$ and $P^{12}(2,~)$.

First of all because all entries of $P^{11}$ are odd, we have $x+2\notin P^{11}$, as $x$ is even. Assume $x+1\in P^{11}$, then $x+1$ is the second smallest element in $P^{11}$, i.e., $x+1=P^{11}_{21}=\sum_{k=d_1-1}^{d_1}\overline{a_k}$, which implies that $\overline{a_{d_1-1}}=x$. 
Then $P^{12}_{21}=\frac{\sum_{k=d_1-1}^{d_2}\overline{a_k}}{2}=\frac{x+4}{2}$. Since $x\ge 4$, we have $\frac{x+4}{2}\le x$. However, by the assertion $\normalsize{\textcircled{\scriptsize{1}}}$, $\frac{x+4}{2}$ should lie in $P^{12}(1,~)$ or $P^{22}(1,~)$. This is a contradiction. Hence neither $x+1$ nor $x+2$ are entries of $P^{11}$.

Further, we claim neither $x+1$ or $x+2$ are entries of $P^{12}(2,~)$. Assume $x+1\in P^{12}(2,~)$, then $x+1$ is the smallest entry in $P^{12}(2,~)$, i.e., $x+1=P^{12}_{21}=\sum_{k=d_1-1}^{d_2}\overline{a_k}$,  which implies that $\overline{a_{d_1-1}}=2x-2$. However, this contradicts that $\overline{a_{d_1-1}}$ is divisible by $12$ (see Proposition \ref{prop2}), because $x$ is even. Assume $x+2\in P^{12}(2,~)$, similarly we have $\overline{a_{d_1-1}}=2x$. Suppose $x=P^{12}_{1,v}=\frac{\sum_{k=d_1}^{d_2+v-1}\overline{a_k}}{2}$, then $\sum_{k=d_1}^{d_2+v-1}\overline{a_k}=2x$. Hence
\begin{align*}
Q^{12}_{d_1-1,v}&=\frac{\sum_{k=d_1-1}^{n-2}\overline{a_k}+\sum_{k=d_2+v}^{n}\overline{a_k}}{2}\\
&=\frac{\sum_{k=d_2}^{n-2}\overline{a_k}+\sum_{k=d_2+1}^{n}\overline{a_k}+\overline{a_{d_1-1}}+2\overline{a_{d_1}}+\overline{a_{d_2}}-\sum_{k=d_1}^{d_2+v-1}\overline{a_k}}{2}\\
&=\frac{\dim X-2+2x+5-2x}{2}=\frac{\dim X+3}{2}.
\end{align*}
However, on the other hand, we have 
$$P^{12}_{d_1,1}=\frac{\sum_{k=1}^{d_2}\overline{a_k}}{2}=\frac{\sum_{k=1}^{d_1}\overline{a_k}+\overline{a_{d_2}}}{2}=\frac{\dim X+3}{2}.$$
Note in the above equalities, we use the assertions in Proposition \ref{prop2} (1). This leads us to the conclusion that $x+2\notin P^{12}(2,~)$, because two identical numbers $Q^{12}_{d_1-1,v}$ and $P^{12}_{d_1,1}$ appear in $T_X^{\lambda}$.\\	

\textbf{Step 2.3:} We claim that neither $x+1$ nor $x+2$ appear as entries of $Q^{12}$.

We prove this claim by contradiction. If the claim is not true, then $x+1$ or $x+2$ must be equal to $Q^{12}_{min}$. So we begin by estimating the value of $Q^{12}_{min}$. Let us write $x$ as $P^{12}_{1,v}$. If $v<n-d_2$, then $\sum_{k=d_1}^{n-2}\overline{a_k}\ge 2x$. It follows that
$$Q^{12}_{min}=\frac{\sum_{k=d_1}^{n-2}\overline{a_k}+\overline{a_n}}{2}\ge\frac{2x+2}{2}=x+1.$$
If $v=n-d_2$, then $\sum_{k=d_1}^{n-1}\overline{a_k}=2x$. It follows that
$$Q^{12}_{min}=\frac{\sum_{k=d_1}^{n-2}\overline{a_k}+\overline{a_n}}{2}=\frac{\sum_{k=d_1}^{n-1}\overline{a_k}+\overline{a_n}-\overline{a_{n-1}}}{2}\ge\frac{2x+4}{2}=x+2.$$
From the above two inequalities, we can derive the following:
\begin{itemize}
	\item If $x+1$ is $Q^{12}_{min}$, then $x=\frac{\sum_{k=d_1}^{n-2}\overline{a_k}}{2}=P^{12}_{1,n-d_2-1}$ and $\overline{a_n}=2$.
	\item If $x+2$ is $Q^{12}_{min}$, then there are three possibilities:
	
	(1) $x=\frac{\sum_{k=d_1}^{n-1}\overline{a_k}}{2}=P^{12}_{1,n-d_2}$ and $\overline{a_n}-\overline{a_{n-1}}=4$,
	
	(2) $x=\frac{\sum_{k=d_1}^{n-2}\overline{a_k}}{2}=P^{12}_{1,n-d_2-1}$ and $\overline{a_n}=4$,
	
	(3) $x+1=\frac{\sum_{k=d_1}^{n-2}\overline{a_k}}{2}=P^{12}_{1,n-d_2-1}$ and $\overline{a_n}=2$.
\end{itemize}
We will show that none of these possibilities exist.

 Suppose $x+1=Q^{12}_{min}$. Then we rule out this possibility by proving $x+2$ does not appear in the associated datum $T_X^{\lambda}$. Assume $x+2$ appears in $T_X^{\lambda}$, then $x+2\in P^{12}$, $Q^{12}$ or $R^1$, as $x+2$ is even. By the proof in Step 2.1, we know $R^1_{min}\ge x+4$ and hence $x+2\notin R^1$. Assume $x+2\in P^{12}$, then according to $x=P^{12}_{1,n-d_2-1}$, we have $x+2=P^{12}_{1,n-d_2}$
 or $x+2=P^{12}_{21}$.
 If the former happens, then $\overline{a_{n-1}}=4$, which contradicts $\overline{a_n}-\overline{a_{n-1}}$  being divisible by $4$. If the latter happens, then $\overline{a_{d_1-1}}=2x$. In this case, we would have $P^{11}_{21}=Q^{12}_{d_1-1,n-d_2}=2x+1$, which is a contradiction. Hence $x+2\notin P^{12}$. Assume $x+2\in Q^{12}$, then according to $x+1=Q^{12}_{min}=Q^{12}_{d_1,n-d_2}$, we have $x+2=Q^{12}_{d_1-1,n-d_2}$
or $x+2=Q^{12}_{d_1,n-d_2-1}$. 
If the former happens, then $\overline{a_{d_1-1}}=2$, which contradicts $\overline{a_{d_1-1}}$ being divisible by $12$. If the latter happens, then $\overline{a_{n-1}}=2$. In this case, we would have $\overline{a_{n-1}}=\overline{a_n}$, which is a contradiction. Hence $x+2\notin Q^{12}$. In summary, $x+2$ does not appear in $T_X^{\lambda}$. \\

Suppose $x+2=Q^{12}_{min}$. Then we first rule out the possibility (1) by calculating the dimension of $X$. On the one hand, according to the possibility (1) and Proposition \ref{prop2}, we have
$$\dim X-2=\sum_{k=d_2}^{n-2}\overline{a_k}+\sum_{k=d_2+1}^{n}\overline{a_k}=2\sum_{k=d_1}^{n-1}\overline{a_k}-2\overline{a_{d_1}}-\overline{a_{d_2}}+\overline{a_n}-\overline{a_{n-1}}=4x-1.$$
Hence $\dim X=4x+1$. On the other hand, from the assertion $\normalsize{\textcircled{\scriptsize{2}}}$ and the equalities $2=P^{12}_{11}$, $x=P^{12}_{1,n-d_2}$, we can deduce that $n-d_2\ge \frac{x}{2}$. This together with $d_1\ge 3$ gives us $\dim X=|T_X^{\lambda}|=2(d_1+1)(n-d_2)+2d_1+\frac{d_1(d_1-1)}{2}\ge 4x+9$. Thus, the possibility (1) cannot happen. We next rule out the possibilities (2) and (3) by proving $x+4$ does not appear in the associated datum $T_X^{\lambda}$. Assume $x+4$ appears in $T_X^{\lambda}$, then $x+4\in P^{12}$, $Q^{12}$ or $R^1$. For the possibilities (2) and (3), since $x\ne P^{12}_{1,n-d_2}$, $R^1_{min}\ge x+5$ by the proof in Step 2.1, hence $x+4\notin R^1$. Further, since $x+2=Q^{12}_{min}=Q^{12}_{d_1,n-d_2}$ and $\overline{a_{d_1-1}}\ge 12$, $Q^{12}_{d_1-1,n-d_2}\ge x+8$. Hence if $x+4\in Q^{12}$, then we must have $x+4=Q^{12}_{d_1,n-d_2-1}$ or $x+4=Q^{12}_{d_1,n-d_2-2}$, $x+3=Q^{12}_{d_1,n-d_2-1}$. If the former happens, then $\overline{a_{n-1}}=4$. If the latter happens, then $\overline{a_{n-1}}=2$. Because $\overline{a_n}$ is $4$ or $2$, whether the former or the latter occurs would contradict to  
 $\overline{a_n}-\overline{a_{n-1}}$  being divisible by $4$.
 Assume $x+4\in P^{12}$, then 
 $x+4=P^{12}_{1,n-d_2}$
or $x+4=P^{12}_{21}$. 
If the former happens, then $\sum_{k=d_1}^{n-1}\overline{a_k}=2x+8$ and $\overline{a_n}-\overline{a_{n-1}}=4$ for both possibilities (2) and (3). Using Proposition \ref{prop2}, we get 
 $\dim X=4x+9$. However, on the other hand, since $x\ne P^{12}_{1,n-d_2}$, we have $n-d_2\ge \frac{x}{2}+1$ and hence 
 $\dim X=|T_X^{\lambda}|\ge 4x+17$.
 This leads to a contradiction. 
If the latter happens, then $\overline{a_{d_1-1}}=2x+4$. Write $x$ as $P^{12}_{1,v}$. Using the similar argument in Step 2.2, we would get two identical numbers $Q^{12}_{d_1-1,v}=P^{12}_{d_1,2}=\frac{\dim X+7}{2}$, which is a contradiction. To sum up, none of the above possibilities exist. Then we are done.
\end{proof}

\textbf{Proof of Theorem \ref{mainthm} for type $D_n$ and $s=2$, $d_s\le n-2$:}	Suppose there exists an Ulrich bundle $E_{\lambda}$ on $X$. Let $T_X^{\lambda}$ be the associated datum  of $E_{\lambda}$. Then any integer between $1$ and $\dim X$ should appear in $T_X^{\lambda}$. In particular, for the integer $x=P^{12}_{1,n-d_2}$, $x+1$ and $x+2$ should appear as entries of $P^{12}(1,~)$ or $P^{22}(1,~)$ by Proposition \ref{prop3}. Since one of the integers $x+1$ and $x+2$ is even, at least one of them would lie in $P^{12}(1,~)$, as every entry in $P^{22}(1,~)$ is odd. However, this is obvious impossible, because $x$ is the largest entry in $P^{12}(1,~)$ by our choice. Therefore we conclude that there are no irreducible homogeneous Ulrich bundles on $X$.

\subsubsection{$J=\{d_1,\ldots,d_s\}$ with $d_s=n-1$ or $d_s=n,~d_{s-1}\ne n-1$ }\label{}
Let $0<d_1,\ldots,d_{s-1}< n-1$ be an increasing sequence of integers.
In type $D_n$, note that $G/P_{\{d_1,\ldots,d_{s-1},n-1\}}$ is isomorphic to $G/P_{\{d_1,\ldots,d_{s-1},n\}}$ as a projective variety.
Therefore, we only consider the existence problem for the case $G$ is of type $D_n$ and $J=\{d_1,\ldots,d_s\}$ with $d_s=n-1$  in this paper.

According to \cite[Section 3.2.1 III Case (b)]{Fang}, the associated datum $T_X^{\lambda}$ of $E_{\lambda}$ is of the following form. (By convention we set $d_0=0$ and $d_{s+1}=n$.)

$T_X^{\lambda}=\{\tilde{P}^{ij}~(1\le i\le j\le s), \tilde{Q}^{ij}~(1\le i\le j\le s,~i\ne s),\tilde{R}^{i}~(1\le i\le s)\}$, where
\begin{align*}
	&\tilde{P}^{ij}_{uv}=\frac{\sum\limits_{k=d_i-u+1}^{d_j+v-1}\overline{a_k}}{j-i+1}~(1\le u\le d_i-d_{i-1},1\le v\le d_{j+1}-d_j);\\
	&\tilde{Q}^{ij}_{uv}=\frac{\sum\limits_{k=d_{i-1}+u}^{n-2}\overline{a_k}+\sum\limits_{k=d_j+v}^{n}\overline{a_k}}{2s-(i+j)}
	~(1\le u\le d_i-d_{i-1},1\le v\le d_{j+1}-d_j);\\
	&\tilde{R}^{i}_{uv}=\frac{\sum\limits_{k=d_{i-1}+u}^{n-2}\overline{a_k}+\sum\limits_{k=d_{i-1}+v}^{n}\overline{a_k}}{2(s-i)+1}~(1\le u< v\le d_i-d_{i-1}).
	\end{align*}
Note that the matrices $\tilde{Q}^{ij}$ appear in $T_X^{\lambda}$ only if $s$ is greater than $1$.

Firstly, in order to show the main theorem, we prepare some lemmas like Lemma \ref{lcm1D} and \ref{lcm2D}.

\begin{lem}\label{lcm3D}
	Let $E_{\lambda}$ be an irreducible homogeneous Ulrich bundle on $X$ with highest weight $\lambda=a_1\lambda_1+\cdots+a_n\lambda_n$. 
	\begin{itemize}
		\item[(i)] For any $1\le k\le d_1-1$, we have
		$lcm(1,2,\ldots, 2s-2)\mid\overline{a_k}$.
		\item[(ii)] For any integer $2\le i\le s$ and $d_{i-1}+1\le k\le d_{i}-1$, we have
		$$lcm(1,2,\ldots, 2s-2i)\mid\overline{a_k}~\text{and}~lcm(2s-2i+2,2s-2i+3,\ldots, 2s-i)\mid\overline{a_k}.$$
		\item[(iii)] If $s\ge 2$, then  $lcm(1,2,\ldots,s)\mid\overline{a_n}$.
	\end{itemize}
\end{lem}
\begin{proof}
	The proof of (i) and (ii) is similar to the proof of Lemma \ref{lcm1D}. 
	For the statement (iii), let us compare the difference 
	$$\tilde{Q}^{i,s-1}_{1,d_s-d_{s-1}}-\tilde{P}^{is}_{d_i-d_{i-1},1}=\frac{\overline{a_n}}{s+1-i}.$$
	Since $E_{\lambda}$ be an Ulrich bundle, the above difference should be an integer and hence $(s+1-i)\mid\overline{a_n}$ for any $1\le i\le s-1$. Then we are done.
\end{proof}

Using Lemma \ref{lcm3D}, we conclude the following immediately. 
\begin{lem}\label{lcm4D}
	With the same assumption as above, then $\overline{a_{d_1}},\overline{a_{d_2}},\ldots,\overline{a_{d_s}}$ are $s$ different odd numbers.
\end{lem}

Next, we will prove the following proposition.
\begin{prop}\label{d_i=i}
	Let $E_\lambda$ be an irreducible homogeneous Ulrich bundle on $X$ with $\lambda=\sum_{i=1}^{n}a_i \lambda_i$. If $s\ge 2$,
	then we have $d_{i}=d_{i-1}+1$ for every $1\leq i \leq s$.
	In particular, we obtain $d_i=i$ for every $1\leq i \leq s$.
\end{prop}
\begin{proof}
Since $E_\lambda$ is an Ulrich bundle, all entries in the associated datum $T_X^{\lambda}$ should be integers. We first determine the parity of each entry of $\tilde{R}^{i}$, $\tilde{P}^{ij}$ and $\tilde{Q}^{ij}$. 
By Lemma \ref{lcm3D}, we know that as long as $s\ge 2$, $\overline{a_k}$ is even for any integer $k$ except $d_1,\ldots,d_s$. By observing  the concrete form of $\tilde{R}^{i}_{uv}$ and using Lemmas \ref{lcm3D} and \ref{lcm4D}, we can easily determine that the numerator of $\tilde{R}^{i}_{uv}$ is odd, which implies that every $\tilde{R}^{i}_{uv}$ is odd, as $\tilde{R}^{i}_{uv}$ is an integer. 
Moreover, for any integer $i$ and $j~(1\le i\le j\le s)$, as long as $j-i$ is even, it is easy to tell by Lemmas \ref{lcm3D} and \ref{lcm4D} that the numerator of $\tilde{P}^{ij}_{uv}$ is odd, and therefore every entry of  $\tilde{P}^{ij}$ is also odd. Similarly, when $j-i$ is odd, every entry of  $\tilde{Q}^{ij}$ is odd. This tells us that the number of odd numbers in $T_X^{\lambda}$ is at least $$\sum\limits_{1\le i\le s}|\tilde{R}^{i}|+\sum\limits_{\substack{1\le i\le j\le s\\j-i~\text{is even}}}|\tilde{P}^{ij}|+\sum\limits_{\substack{1\le i\le j\le s\\j-i~\text{is odd}}}|\tilde{Q}^{ij}|,$$ where $|M|$ represents the number of entries in a matrix $M$. It follows that the number of even numbers in $T_X^{\lambda}$ is at most
$$\sum\limits_{\substack{1\le i\le j\le s\\j-i~\text{is odd}}}|\tilde{P}^{ij}|+\sum\limits_{\substack{1\le i\le j\le s,i\ne s\\j-i~\text{is even}}}|\tilde{Q}^{ij}|.$$
 Note that for any integer $i$ and $j~(1\le i\le j\le s)$, $|\tilde{P}^{ij}|$ is always equal to $|\tilde{Q}^{ij}|$. So the number of odd numbers in $T_X^{\lambda}$ minus the number of even numbers is greater than or equal to 
 $$
 \sum_{1\le i\le s}|\tilde{R}^{i}|+|\tilde{P}^{ss}|=\sum_{i=1}^{s}\frac{(d_i-d_{i-1})(d_i-d_{i-1}-1)}{2}+d_s-d_{s-1}\ge 1.
 $$
On the other hand, since $E_\lambda$ is an Ulrich bundle,  $T_X^{\lambda}=\{1,2,\ldots, \dim X\}$. Hence the number of odd numbers in $T_X^{\lambda}$ minus the number of even numbers is at most $1$. This together with the above equality tells us $d_{i}=d_{i-1}+1$ for every $1\leq i \leq s$.
\end{proof}

From Proposition \ref{d_i=i}, we can see that the matrix $\tilde{R}^{i}$ is empty for any $1\le i\le s$. In addition, $\tilde{P}^{ij}$ and $\tilde{Q}^{ij}$ are both $(1\times 1)$-matrices with only one entry $\tilde{P}^{ij}_{11}=\frac{\sum_{k=i}^{j}\overline{a_k}}{j-i+1}$ and $\tilde{Q}^{ij}_{11}=\frac{\sum_{k=i}^{n-2}\overline{a_k}+\sum_{k=j+1}^{n}\overline{a_k}}{2s-(i+j)}$ respectively. And since $d_s=n-1$ and $d_s=s$ by Proposition \ref{d_i=i}, we have $s=n-1$. Note that for type $D_n$, $n$ is naturally greater than or equal to $4$ and hence $s\ge 3$.\\

\textbf{Proof of Theorem \ref{mainthm} for type $D_n$ and $d_s=n-1$:}		
Suppose $E_\lambda$ is an Ulrich bundle on $X$. We first estimate the value of $\tilde{Q}^{ij}_{11}$ in terms of $s$. If $s\ge 5$, then according to Lemma \ref{lcm3D} (iii), $\overline{a_n}\ge 2s(s-1)$. Hence $\tilde{Q}^{ij}_{11}\ge \frac{\overline{a_n}}{2s-2}\ge s$, which is great than $4$. If $s$ is equal to $3$ or $4$, then $\overline{a_n}\ge s(s-1)$. By Lemma \ref{lcm4D}, $\overline{a_{d_1}},\overline{a_{d_2}},\ldots,\overline{a_{d_s}}$ are $s$ different odd numbers. Therefore $$\tilde{Q}^{ij}_{11}>\frac{(s-i)^2+(s-j)^2+s(s-1)}{2s-(i+j)}=j-i+\frac{2(s-j)^2+s(s-1)}{2s-(i+j)}.$$ If $i=j$, then the latter is greater than or equal to $2\sqrt{\frac{s(s-1)}{2}}$. If $j>i$, then the latter is greater than $1+\frac{s}{2}$. Therefore, when $s=3$ or $s=4$, we always have $\tilde{Q}^{ij}_{11}>2$. It follows that $2$ can only appear as $\tilde{P}^{t_0,t_0+1}_{11}$ for some integer $t_0~(1\le t_0\le s-1)$. This implies that either $\overline{a_{t_0}}=1,~\overline{a_{t_0+1}}=3$ or $\overline{a_{t_0}}=3,~\overline{a_{t_0+1}}=1$. Applying almost verbatim the proof after Step 2 of Section \ref{section1}, we get $4$ does not appear in any matrix $\tilde{P}^{ij}$. On the other hand, from the above analysis, we see that only when $s$ is equal to $3$ or $4$, $4$ can appear as $\tilde{Q}^{ij}_{11}$ for some pair $(i,j)$, where $1\le i\le j\le s~(i\ne s)$ and $j-i$ is even. Hence we have the inequality $4>\frac{(s-i)^2+(s-j)^2+s(s-1)}{2s-(i+j)}$. Substituting $s=3$ or $s=4$ and the possible pairs $(i,j)$ into this inequality, we find that only $s=3$ and $(i,j)=(1,1)$ satisfy this inequality. That is to say that $4$ is equal to 
$\tilde{Q}^{11}_{11}=\frac{\sum_{k=1}^{2}\overline{a_k}+\sum_{k=2}^{4}\overline{a_k}}{4}$, which means that $\overline{a_4}=6$, $\overline{a_2}=1$ and $\overline{a_1}+\overline{a_3}=8$. It leads to that $\tilde{P}^{13}_{11}=\frac{\sum_{k=1}^{3}\overline{a_k}}{3}=3$, which coincides with $\overline{a_{t_0}}$ or $\overline{a_{t_0+1}}$. In summary, we conclude that there are no irreducible homogeneous Ulrich bundles on $X$.
  
\subsubsection{$J=\{d_1,\ldots,d_s\}$ with $d_{s-1}=n-1$ and $d_s=n$}\label{}

First we notice that when $J$ is the set $\{d_1,\ldots,d_s\}$ with $d_{s-1}=n-1$ and $d_s=n$, the Picard number of $X=G/P_J$ is naturally greater than or equal to $2$. According to \cite[Section 3.2.1 III Case (d)]{Fang}, the associated datum $T_X^{\lambda}$ of $E_{\lambda}$ is of the following form. 

$T_X^{\lambda}=\{\hat{P}^{ij}, \hat{Q}^{ij}~(1\le i\le j\le s-1),\hat{R}^{i}~(1\le i\le s-1)\}$, where
\begin{align*}
	&\hat{P}^{ij}_{uv}=\frac{\sum\limits_{k=d_i-u+1}^{d_j+v-1}\overline{a_k}}{j-i+1}~(1\le u\le d_i-d_{i-1},1\le v\le d_{j+1}-d_j);\\
	&\hat{Q}^{ij}_{uv}=\frac{\sum\limits_{k=d_{i-1}+u}^{n-2}\overline{a_k}+\sum\limits_{k=d_j+v}^{n}\overline{a_k}}{2s-1-(i+j)}
	~(1\le u\le d_i-d_{i-1},1\le v\le d_{j+1}-d_j);\\
	&\hat{R}^{i}_{uv}=\frac{\sum\limits_{k=d_{i-1}+u}^{n-2}\overline{a_k}+\sum\limits_{k=d_{i-1}+v}^{n}\overline{a_k}}{2(s-i)}~(1\le u< v\le d_i-d_{i-1}).
\end{align*}
First of all, let us prepare the following lemma.

\begin{lem}\label{lcm5D}
	Let $E_{\lambda}$ be an irreducible homogeneous Ulrich bundle on $X$ with highest weight $\lambda=a_1\lambda_1+\cdots+a_n\lambda_n$. 
	\begin{itemize}
		\item[(i)] For any $1\le k\le d_1-1$, we have
		$lcm(1,2,\ldots, 2s-3)\mid\overline{a_k}$.
		\item[(ii)] For any integer $2\le i\le s-1$ and $d_{i-1}+1\le k\le d_{i}-1$, we have
		$$lcm(1,2,\ldots, 2s-2i-1)\mid\overline{a_k}~\text{and}~lcm(2s-2i+1,2s-2i+2,\ldots, 2s-i-1)\mid\overline{a_k}.$$
		\item[(iii)]  $lcm(1,2,\ldots,s-1)\mid\overline{a_n}-\overline{a_{n-1}}$.
	\end{itemize}
\end{lem}
\begin{proof}
	The proof of (i) and (ii) is similar to the proof of Lemma \ref{lcm1D}. 
	For the statement (iii), let us compare the difference 
	$$\hat{Q}^{i,s-1}_{1,d_s-d_{s-1}}-\hat{P}^{i,s-1}_{d_i-d_{i-1},1}=\frac{\overline{a_n}-\overline{a_{n-1}}}{s-i}.$$
	Since $E_{\lambda}$ is an Ulrich bundle, the above difference should be an integer and hence $(s-i)\mid\overline{a_n}-\overline{a_{n-1}}$ for any $1\le i\le s-1$. Then we are done.
\end{proof}
\begin{rem}\label{remark}
From Lemma \ref{lcm5D}, it is easy to find that when $s$ is greater than or equal to $4$, $\overline{a_k}$ is even for any integer $1\le i\le s-1$ and $k\in[d_{i-1}+1,d_{i}-1]$. 

When $s$ is equal to $3$, $\overline{a_k}$ is even for any integer $1\le k\le d_1-1$ by Lemma \ref{lcm5D} (i). If $d_2-d_1\ge 2$, then 
$\hat{R}^{2}$ is not empty. Since for any $1\le u<d_2-d_1$,
$\hat{R}^{2}_{u,d_2-d_1}=\frac{\sum_{k=d_1+u}^{d_2-1}\overline{a_k}+\overline{a_{d_2}}+\overline{a_{d_3}}}{2}$ is an integer and $\overline{a_{d_2}}$ has the same parity as $\overline{a_{d_3}}$ by Lemma \ref{lcm5D} (iii), $\sum_{k=d_1+u}^{d_2-1}\overline{a_k}$ is even for any $1\le u<d_2-d_1$. It follows that $\overline{a_k}$ is even for any integer $d_1+1\le k\le d_2-1$. 

 When $s$ is equal to $2$, since $d_1=n-1$ and $n\ge 4$, we have $d_1\ge 3$. By comparing the difference between any two adjacent elements of the first row and last column of $\hat{R}^1$, we can see that all  $\overline{a_k}~(1\le k\le d_1-1)$ are even. In summary, as long as $s$ is greater than or equal to 2, we always have $\overline{a_k}$ is even for any integer $1\le i\le s-1$ and $k\in[d_{i-1}+1,d_{i}-1]$. 
\end{rem}

Using Lemma \ref{lcm5D} and Remark \ref{remark}, we obtain the non-trivial result.
\begin{lem}\label{lcm6D}
	With the same assumption as Lemma \ref{lcm5D}, then $\overline{a_{d_1}},\overline{a_{d_2}},\ldots,\overline{a_{d_s}}$ are $s$ different odd numbers.
\end{lem}
\begin{proof}
First we notice that  $\overline{a_{d_i}}=\hat{P}^{ii}_{min}~(1\le i\le s-1)$ and $\overline{a_{d_s}}=\hat{Q}^{s-1,s-1}_{min}$. Because $E_{\lambda}$ is an Ulrich bundle, $\overline{a_{d_i}}~(1\le i\le s)$ are $s$ different positive integers. Moreover, since $E_{\lambda}$ is an Ulrich bundle, we have
 $\min_{1\le i\le s}\{\overline{a_{d_i}}\}=1$ by \cite[Remark 3.4]{Fang}. So to prove this lemma it suffices to show that $\overline{a_{d_i}}~(1\le i\le s)$ have the same parity.	
	By Remark \ref{remark}, we see that  $\overline{a_k}$ is even for any integer $1\le i\le s-1$ and $k\in[d_{i-1}+1,d_{i}-1]$. Since $E_{\lambda}$ is an Ulrich bundle, $\hat{P}^{i,i+1}_{min}$ is an integer, which implies that all $\overline{a_{d_i}}~(1\le i\le s-1)$ have the same parity. In addition, Lemma \ref{lcm5D} (iii) tells us that as long as $s\ge 3$, $\overline{a_{d_s}}(=\overline{a_n})$ and $\overline{a_{d_{s-1}}}(=\overline{a_{n-1}})$ have the same parity. Therefore, if $s\ge 3$, then $\overline{a_{d_1}},\overline{a_{d_2}},\ldots,\overline{a_{d_s}}$ have the same parity. For the case $s$ is equal to $2$, because $\hat{R}^1_{1,d_1}=\frac{\sum_{k=1}^{d_1-1}\overline{a_k}+\overline{a_{d_1}}+\overline{a_{d_2}}}{2}$ is an integer and $\overline{a_k}$ is even for any $1\le k\le d_1-1$ by Remark \ref{remark}, $\overline{a_{d_1}}$ and $\overline{a_{d_2}}$ naturally have the same parity.
\end{proof}

As a consequence of Lemmas \ref{lcm5D} and \ref{lcm6D}, we obtain the following.

\begin{cor}\label{R^i}
With the same assumption as Lemma \ref{lcm5D}. If the matrix $\hat{R}^{i}$ is not empty, then every entry of $\hat{R}^{i}$ is greater than $2$. In particular, if $s\ge 4$, then every entry of $\hat{R}^{i}$ is greater than $4$. 
\end{cor}
\begin{proof}
Note that the matrix $\hat{R}^{i}$ is not empty is equivalent to say $d_i-d_{i-1}\ge 2$. 
By Lemma \ref{lcm6D}, for any $1\le i\le s-1$ and $1\le u< v\le d_i-d_{i-1}$, we have
\begin{align*}
	\hat{R}^{i}_{uv}&=\frac{\sum\limits_{k=d_{i-1}+u}^{n-2}\overline{a_k}+\sum\limits_{k=d_{i-1}+v}^{n}\overline{a_k}}{2(s-i)}>\frac{\sum\limits_{t=i}^{s-2}\overline{a_{d_t}}+\sum\limits_{t=i}^{s}\overline{a_{d_t}}}{2(s-i)}\\
	&\ge \frac{(s-1-i)^2+(s+1-i)^2}{2(s-i)}=2+\frac{(s-1-i)^2}{(s-i)}.
\end{align*}
Therefore, every entry of $\hat{R}^{i}$ is greater than $2$. If moreover $s\ge 4$, then by Lemma \ref{lcm5D} (i) and (ii), we have $\overline{a_k}\ge s(s-1)$ for any integer $1\le i\le s-1$ and $k\in[d_{i-1}+1,d_{i}-1]$. Thus every entry $\hat{R}^{i}_{uv}$ satisfies
\begin{align*}
\hat{R}^{i}_{uv}&\ge\frac{\sum\limits_{t=i}^{s-2}\overline{a_{d_t}}+\sum\limits_{t=i}^{s}\overline{a_{d_t}}+s(s-1)}{2(s-i)}\ge \frac{(s-1-i)^2+(s+1-i)^2+s(s-1)}{2(s-i)}\\
&\ge 2\sqrt{1+\frac{s(s-1)}{2}}>4.
\end{align*}
\end{proof}

Before proving the main theorem, let's prove a simple case.
\begin{prop}\label{s=2,3}
 If $s$ is $2$ or $3$, then there are no irreducible homogeneous Ulrich bundles on $X$ with respect to the minimal ample class.
\end{prop}
\begin{proof}
Suppose there exists an Ulrich bundle $E_{\lambda}$ on $X$. Let $T_X^{\lambda}$ be the associated datum of $E_{\lambda}$.  If $s$ is $2$, then $T_X^{\lambda}$ consists of  $\hat{P}^{11}$, $\hat{Q}^{11}$ and $\hat{R}^{1}$. 
By Remark \ref{remark} and Lemma \ref{lcm6D}, we see that $\overline{a_1},\ldots, \overline{a_{n-2}}$ are even and $\overline{a_{n-1}}$, $\overline{a_n}$ are odd. Thus all entries of $\hat{P}^{11}$ and $\hat{Q}^{11}$ are odd. Hence $2$ can only appear in $\hat{R}^{1}$. However, by Corollary \ref{R^i}, we know it is impossible.\\  

If $s$ is $3$, then $T_X^{\lambda}$ consists of  $\hat{P}^{ij}$, $\hat{Q}^{ij}~(1\le i\le j\le 2)$ and $\hat{R}^{i}~(i=1,2)$. Similarly, from Remark \ref{remark} and Lemma \ref{lcm6D}, we can easily judge that all entries of $\hat{P}^{ii}$ and $\hat{Q}^{ii}~(i=1,2)$ are odd. Morover, since every entry of $\hat{R}^{i}$ is greater than $2$ by Corollary \ref{R^i}, $2$ can only appear as the smallest entry of $\hat{P}^{12}$ or $\hat{Q}^{12}$. 

If $2$ is equal to $\hat{P}^{12}_{min}$, then $d_1=d_2-1=n-2$ and $(\overline{a_{d_1}}, \overline{a_{d_2}})=(1,3)$ or $(3,1)$. In this case, $\hat{R}^2$ is empty. In addition, the second smallest element in $\hat{P}^{12}$ is $\hat{P}^{12}_{21}(=\frac{\sum_{k=d_1-1}^{d_2}\overline{a_k}}{2})$, which is greater than $4$, as $\overline{a_{d_1-1}}\ge 6$ by Lemma \ref{lcm5D} (i). Hence the  candidate of $4$ can only be  $\hat{Q}^{12}_{min}$ or $\hat{R}^1_{min}$. If $4=\hat{Q}^{12}_{min}=\frac{\overline{a_{d_1}}+\overline{a_n}}{2}$, then $\overline{a_n}$ is $5$ or $7$, because $\overline{a_{d_1}}$ is $1$ or $3$. This results in $\hat{Q}^{11}_{min}=\frac{\sum_{k=n-2}^{n}\overline{a_k}}{3}$ being either equal to $3$ or not an integer. However, in either case, it contradicts the assumption that $E_{\lambda}$ is an Ulrich bundle. If $4=\hat{R}^1_{min}=\frac{\sum_{k=d_1-1}^{n-2}\overline{a_k}+\sum_{k=d_1}^{n}\overline{a_k}}{4}$, then we must have $\overline{a_{d_1-1}}=6$,
 $\overline{a_{d_1}}=1$, $\overline{a_{d_2}}=3$ and $\overline{a_n}=5$. It follows that $\hat{Q}^{11}_{min}$ is equal to $3$, this leads to a contradiction.
 
 If $2$ is equal to $\hat{Q}^{12}_{min}$, then $d_1=n-2$ and $(\overline{a_{d_1}}, \overline{a_n})=(1,3)$ or $(3,1)$. In this case, $\hat{R}^2$ is also empty. In addition, since $\overline{a_{d_1-1}}\ge 6$, the second smallest element $\hat{Q}^{12}_{d_1-1,1}$ in $\hat{Q}^{12}$ is greater than $4$. Hence the candidate of $4$ can only be $\hat{P}^{12}_{min}$ or $\hat{R}^1_{min}$. If $4=\hat{P}^{12}_{min}$, then $\overline{a_{d_2}}$ is $5$ or $7$. This results in $\hat{Q}^{11}_{min}$ being either equal to $3$ or not an integer, which is impossible. If $4=\hat{R}^1_{min}$, then we must have $\overline{a_{d_1-1}}=6$,
 $\overline{a_{d_1}}=1$, $\overline{a_{d_2}}=5$ and $\overline{a_n}=3$. This makes $\hat{Q}^{11}_{min}$ equal to $3$, contrary to hypothesis.
\end{proof}

\textbf{Proof of Theorem \ref{mainthm} for type $D_n$ and $d_{s-1}=n-1,d_s=n$:} In Proposition \ref{s=2,3}, we have proved the main theorem for the cases $s=2$ and $s=3$. From now on, we assume that $s$ is greater than or equal to $4$ and assume there exists an Ulrich bundle $E_{\lambda}$ on $X$. From Corollary \ref{R^i}, we first determine that $2$ and $4$ can only be entries of $\hat{P}^{ij}$ or $\hat{Q}^{ij}$. \\

\textbf{Case 1:} If there is a pair of integers $(i,j)$ such that
$2=\hat{P}^{ij}_{min}$,
we can infer that $(i,j)=(t,t+1)$ for some $1\leq t\leq s-2$, $d_{t+1}=d_t+1$ and $\overline{a_{d_t}}+\overline{a_{d_{t+1}}}=4$. Further, we can infer that $\overline{a_n}\ge 5$, since $\overline{a_n}$ is different from $\overline{a_{d_i}}$ for $1\le i\le s-1$.
By the same argument as Step $3$ in Section \ref{section1}, we get $4$ would not appear as an entry of $\hat{P}^{ij}$.
Therefore, $4$ appears as the smallest entry of $\hat{Q}^{ij}$ for some $1\leq i\leq j\leq s-1$. By simple calculation, we find that every entry in $\hat{Q}^{ij}$ satisfies the following inequality:
\begin{align}\label{inequality}
\hat{Q}^{ij}_{uv}\ge	\left\{
	\begin{array}{ll}
		\frac{(s-i)^2}{s-i}=s-i, & \text{if}~j=s-1;\\
		\frac{(s-i+1)^2+(s-j-2)^2}{2s-1-(i+j)}=j-i+3+\frac{2(s-j-2)^2}{2s-1-(i+j)}, & \text{if}~j\le s-2.\\
	\end{array}
	\right.   
\end{align}
Moreover, note that when $j-i$ is even, every entry of $\hat{Q}^{ij}$ is odd. Hence from the inequality (\ref{inequality}), we can infer that if
$4$ lies in $\hat{Q}^{ij}$, then the candidate for the pair $(i,j)$ is either $(s-2,s-1)$, $(s-3,s-2)$ or $(s-4,s-1)$ and $4=\hat{Q}^{ij}_{min}=\frac{\sum_{k=d_i}^{n-2}\overline{a_k}+\sum_{k=d_{j+1}}^{n}\overline{a_k}}{2s-1-(i+j)}$.\\

\textbf{Case 1.1:} Suppose $(i,j)$ is $(s-2,s-1)$. Then from the above equality and the assertion $\overline{a_k}\ge s(s-1)$ for any $k\ne d_1, d_2,\ldots,d_s$ (see Lemma \ref{lcm5D} (i) and (ii)), we can deduce that $d_{s-2}=n-2$ and $\overline{a_{d_{s-2}}}+\overline{a_n}=8$. Since $\overline{a_n}\ge 5$, we have $\overline{a_{d_{s-2}}}$ is $1$ or $3$. This implies that $t$ is $s-3$ or $s-2$. Hence  
$\overline{a_{d_{s-3}}}+\overline{a_{d_{s-2}}}=4$ or $\overline{a_{d_{s-2}}}+\overline{a_{d_{s-1}}}=4$. If the former happens, then $$\hat{Q}^{s-3,s-1}_{min}=\frac{\overline{a_{d_{s-3}}}+\overline{a_{d_{s-2}}}+\overline{a_n}}{3}=\frac{4+\overline{a_n}}{3}.$$
  If the latter happens, then 
$$\hat{Q}^{s-2,s-2}_{min}=\frac{\overline{a_{d_{s-2}}}+\overline{a_{d_{s-1}}}+\overline{a_n}}{3}=\frac{4+\overline{a_n}}{3}.$$ Since $\overline{a_n}$ is $5$ or $7$, $\frac{4+\overline{a_n}}{3}$ is either equal to $3$ or not an integer, contrary to hypothesis. \\

\textbf{Case 1.2:} Suppose $(i,j)$ is $(s-3,s-2)$. Then for the same reason as in Case 1.1, we can infer that $d_{s-l}=n-l$ for any $0\le l\le 3$ and $\overline{a_{d_{s-3}}},~\overline{a_{d_{s-2}}},~\overline{a_{d_{s-1}}},~\overline{a_n}$ are $1,~3,~5,~7$ up to a permutation. It implies that $t=s-3$ or $t=s-2$. Thus $\overline{a_{d_{s-2}}}$ is $1$ or $3$. However, in either case $$\hat{Q}^{s-3,s-3}_{min}=\frac{\sum_{k=d_{s-3}}^{n-2}\overline{a_k}+\sum_{k=d_{s-2}}^{n}\overline{a_k}}{5}=\frac{16+\overline{a_{d_{s-2}}}}{5}$$ is not an integer. This is a contradiction.\\

\textbf{Case 1.3:} Suppose $(i,j)$ is $(s-4,s-1)$. Then for the same reason as in Case 1.2, we can infer that $d_{s-l}=n-l$ for any $0\le l\le 4$ and $\overline{a_{d_{s-4}}},~\overline{a_{d_{s-3}}},~\overline{a_{d_{s-2}}},~\overline{a_n}$ are $1,~3,~5,~7$ up to a permutation. It implies that $t=s-4$ or $t=s-3$. If $t=s-4$, then $\overline{a_{d_{s-4}}}+\overline{a_{d_{s-3}}}=4$ and $\overline{a_{d_{s-2}}}+\overline{a_n}=12$. Thus  $\hat{Q}^{s-3,s-1}_{min}=\frac{12+\overline{a_{d_{s-3}}}}{3}$ is either equal to $5$ or not an integer, contrary to hypothesis. If $t=s-3$, then $\overline{a_{d_{s-3}}}+\overline{a_{d_{s-2}}}=4$ and $\overline{a_{d_{s-4}}}+\overline{a_n}=12$, which implies that $\overline{a_n}=5$ or $\overline{a_n}=7$. Similarly, we have    
 $\hat{Q}^{s-3,s-1}_{min}=\frac{4+\overline{a_n}}{3}$ is either equal to $3$ or not an integer, contrary to hypothesis. \\  

\textbf{Case 2:} If there is a pair of integers $(i,j)$ such that
$2=\hat{Q}^{ij}_{min}$,
We can infer that $(i,j)$ can only be $(s-2,s-1)$ according to the inequality (\ref{inequality}). Then we have $d_{s-2}=n-2$ and $(\overline{a_{d_{s-2}}}, \overline{a_n})=(1,3)$ or $(3,1)$. Now we look for a candidate element in $\hat{Q}^{ij}$ or $\hat{P}^{ij}$ that should be $4$. \\

\textbf{Case 2.1:}
Suppose $4$ lies in some $\hat{Q}^{ij}$. Then the possible choice for $(i,j)$ is $(s-3,s-2)$ or $(s-4,s-1)$ by the inequality (\ref{inequality}). In either case, using the equality $\overline{a_{d_{s-2}}}+\overline{a_n}=4$ and a similar argument in Case 1, we can infer that $\overline{a_{d_{s-3}}}$ is either $5$ or $7$. Then $\hat{Q}^{s-3,s-1}_{min}=\frac{4+\overline{a_{d_{s-3}}}}{3}$ is either equal to $3$ or not an integer, contrary to hypothesis.\\

\textbf{Case 2.2:}
Suppose $4$ lies in some $\hat{P}^{ij}$. First, we note that when $j-i$ is even, every entry of $\hat{P}^{ij}$ is odd. Hence $4\in \hat{P}^{ij}$ implies $j-i$ is odd. On the other hand, since $(\overline{a_{d_{s-2}}}, \overline{a_n})=(1,3)$ or $(3,1)$, we have $\overline{a_{d_l}}\ge 5$ for any $1\le l\le s$ and $l\ne s-2,s$ by Lemma \ref{lcm6D}. Hence when $j\le s-3$, we get $$\hat{P}^{ij}_{uv}\ge\frac{\sum\limits_{l=i}^{j}\overline{a_{d_l}}}{j-i+1}\ge\frac{5+7+\cdots+2(j-i)+5}{j-i+1}=5+j-i\ge5.$$
And when $s-2\le j\le s-1$ and $j-i\ge 2$, we have 
$$\hat{P}^{ij}_{uv}\ge \frac{\sum\limits_{l=i}^{j}\overline{a_{d_l}}}{j-i+1}\ge\frac{1+5+7+\cdots+2(j-i)+3}{j-i+1}=4+\frac{(j-i)^2-3}{j-i+1}>4.$$
 Hence from the assumption $4\in \hat{P}^{ij}$, we can derive that $(i,j)$ is $(s-3,s-2)$ or $(s-2,s-1)$ and $4=\hat{P}^{ij}_{min}=\frac{\sum_{k=d_i}^{d_j}\overline{a_k}}{j-i+1}$. If $(i,j)=(s-3,s-2)$, then $\sum_{k=d_{s-3}}^{d_{s-2}}\overline{a_k}=8$ and hence 
$\hat{Q}^{s-3,s-1}_{min}=\frac{8+\overline{a_n}}{3}$. If $(i,j)=(s-2,s-1)$, then $\sum_{k=d_{s-2}}^{d_{s-1}}\overline{a_k}=8$ and hence $\hat{Q}^{s-2,s-2}_{min}=\frac{8+\overline{a_n}}{3}$. Because $\overline{a_n}$ is $1$ or $3$, $\frac{8+\overline{a_n}}{3}$ is either equal to $3$ or not an integer, contrary to hypothesis. In summary, we complete the proof.

\bibliography{ref}

\end{document}